\documentclass[11pt]{article}
\usepackage{amsmath}
\usepackage{amssymb}
\usepackage{amsthm}
\usepackage{enumerate}
\usepackage[usenames,dvipsnames]{pstricks}
\usepackage{epsfig}
\usepackage{pst-grad}
\usepackage{pst-plot} 
\usepackage[a4paper]{geometry}
\theoremstyle{plain}

\newtheorem{thm}{Theorem}
\newtheorem{lem}{Lemma}
\newtheorem{prop}{Proposition}
\newtheorem{cor}{Corollary}
\newtheorem{fact}{Fact}

\theoremstyle{definition}

\theoremstyle{plain}
\newtheorem{Ass}{Assumption}
\newtheorem{rmq}{Remark}

\newcommand{\eps}{\ensuremath{\varepsilon}} 
\newcommand{\Op}{\ensuremath{\mathcal{L}}}
\newcommand{\Hyp}{\ensuremath{\mathbb{H}}}
\newcommand{\R}{\ensuremath{\mathbb{R}}}
\newcommand{\Prob}{\ensuremath{\mathbb{P}}}
\newcommand{\M}{\ensuremath{\mathcal{M}^{n}}}
\newcommand{\B}{\ensuremath{\mathcal{B}}}
\newcommand{\ind}{\ensuremath{\mathbf{1}}}
\newcommand{\E}{\ensuremath{\mathbb{E}}}

\begin{document}

\title{A Poincar\'e cone condition in the Poincar\'e group}
\author{Tardif Camille \footnote{\texttt{ctardif@math.unistra.fr}. Institut de Recherche Math\'ematique Avanc\'ee. Universit\'e de Strasbourg.} \\  }
\date{}

\maketitle

\abstract{ In \cite{B-A_Grad}, Ben Arous and Gradinaru described the singularity of the Green function of a general sub-elliptic diffusion. In this article we first adapt their proof to the more general context of a hypoelliptic diffusion. In a second time, we deduce a Wiener criterion and a Poincar\'e cone condition for Dudley's diffusion.}

\tableofcontents
\section{Introduction}
In \cite{Dud66}, Dudley introduced a relativistic Brownian motion whose trajectories are time-like and which is invariant in law under Lorentz transformations. The space of states of this relativistic Brownian motion is the unit tangent bundle of the Minkowski spacetime, and the diffusion consists in a Brownian motion in $\Hyp^{d}$ (the hyperboloid model of the hyperbolic space) and its time integral in Minkowski spacetime $\R^{1,d}$. In the spirit of the Eells-Elworthy construction of the Brownian motion on $\Hyp^{d}$, Dudley's diffusion can be obtained by projecting a diffusion $(g_{t},\xi_{t})$ in the orthonormal frame bundle of $\R^{1,d}$. The point $\xi_{t}$ belongs to $\R^{1,d}$, while $g_{t}$ is an orthonormal frame at point $\xi_{t}$ for Minkowski metric. It is fortunate that this bundle has a group structure and the latter diffusion can be chosen to be left invariant in that group, with a generator of the form 
\begin{align}
\Op= \frac{\sigma^{2}}{2} \sum_{i=1}^{d} V_{i}^{2} + H_{0}, \label{operateur}
\end{align}
 satisfying the weak H\"ormander hypoellipticity condition (this means, in particular, that the drift $H_{0}$ is needed to generate a Lie algebra of maximal rank).

Intuitively, $\{ (g_{t}, \xi_{t}) \}_{t\geq 0}$ describes the timelike trajectory of a small rigid object in Minkowski spacetime and consists in a stochastic perturbation of a geodesic trajectory by random perturbation of its velocity. 
This Lorentzian analogue to the Euclidian Brownian motion, was studied by Bailleul in \cite{Bailleul-Raugi} where he determined its Poisson boundary by providing a comprehensive description of asymptotic behaviour. This study was completed in \cite{Tard12} where the Lyapunov spectrum of its flow was described. In the present work we are interested in describing, by means of a (local) Wiener criterion, the thinness of sets with respect to this diffusion. This kind of result is based on the knowledge of the geometry of the level sets of the Green function, and more generally on the knowledge of its behaviour on the diagonal.

For the case of a generic sub-elliptic diffusion with H\"ormander type generator 
\[
\frac{1}{2} \sum_{i=1}^{m} X_{i}^{2}
\]
Ben Arous and Gradinaru (\cite{B-A_Grad}) showed, following Nagel, Stein and Wainger (\cite{N-S-W}) but using stochastic tools, that the Green function is equivalent on the diagonal to the Green function of a Brownian motion on a free nilpotent Lie group. The estimates of the Green function they obtained allowed them to write a general Wiener criterion for such diffusions. 

The diffusion studied in this article, generated by \eqref{operateur}, is of parabolic type. The Green function is not explicit and we don't have estimates which allow us to describe the geometry of its level sets. Contrary to the previous sub-elliptic situation, results about estimates of general parabolic Green functions are not  known. We find in the literature some studies of particular examples. In \cite{Uchi89}, Uchiyama obtained a probabilistic proof of a Wiener criterion for the heat operator based on a fine study of the level sets of the Green function (which has an explicit expression in that case). More general results have been obtained by Garofalo \& al. ( \cite{GaroLanco88}, \cite{GaroSega90}, \cite{GaroLanco90}) in different parabolic  cases ( heat operator with variable coefficients in $\R^{n}$, heat operator in the Heisenberg group and Kolmogorov parabolic operator with constant coefficients). The proofs are based on either explicit expression of Green functions or strong Gaussian estimates. Recently, Menozzi \& al. (\cite{DelMen10}, \cite{KonMenMolch11}, \cite{Men11}) obtained, by a parametrix method, some Gaussian estimates of the Green function for general Kolmogorov parabolic type operators. Unfortunately their results cannot be used directly for Dudley's diffusion.  

In this article, we provide a weak Wiener criterion for $\{(g_{t},\xi_{t})\}_{t \geq 0}$, which concerns sets into some homogeneous cones and we deduce a Poincar\'e cone sufficient condition for thinness. For this we first refine the result of Ben Arous and Gradinaru \cite{B-A_Grad} to the case of generic diffusions with H\"ormander type generator
\[
\frac{1}{2}\sum_{i=1}^{m} X_{i}^{2} +X_{0}
\]
under the condition that the vector space spanned by the Lie brackets of at most $r$ of the vectors $X_{1}, \dots, X_{m}$ and $X_{0}$ have full rank everywhere. We prove in Section \ref{sec2} (Theorem \ref{thm1}) that the Green function of a such hypoelliptic diffusion is equivalent on the diagonal to the Green function a Brownian motion with a drift on a $r$-step free nilpotent Lie group. Some natural dilations on that group make it possible to express the Green function of this model situation in simple term. Theorem \ref{thm1} is far from giving a full description of the behaviour of the Green function on the diagonal for a general hypoelliptic diffusion, nevertheless we use it to obtain in Section \ref{sec3} a weak Wiener criterion and a Poincar\'e cone condition for $\{(g_{t},\xi_{t})\}_{t\geq0}$ in the Poincar\'e group.

\section{Estimates of the Green function for a general hypoelliptic diffusion}\label{sec2}

Let $X_{0}, X_{1}, \dots, X_{m}$ be smooth vector fields on a smooth connected manifold $\mathcal{M}$ of dimension $n \geq 3$ such that the Lie algebra generated by $X_{0}, X_{1}, \dots, X_{m}$has full rank everywhere:
\[
\forall x \in \mathcal{M}, \quad Lie(X_{0}, X_{1}, \dots, X_{m}) \vert_{x} = T_{x} \mathcal{M}. \label{H} \tag{H}
\]
Under this condition, H\"ormander \cite{Hor63} showed the hypoellipticity of the operator 
\[
\Op:= \frac{1}{2} \sum_{i=1}^{m} X_{i}^{2} + X_{0},
\]
and Bony (\cite{Bony69}) proved the existence of the Green function $G_{U}(\cdot, \cdot)$ for smooth bounded domain $U \subset \mathcal{M}$. Given $f\in C^{0}(U)$, the unique function $\phi \in C^{0}( \bar{U})$ such that 
\begin{align*}
\Op \phi &=f \quad \mathrm{on } \  U\\
 \phi&= 0 \quad \mathrm{on } \ \partial U,
\end{align*}
holds in the sense of distributions is given by
\[
\phi(x) = \int_{U} G_{U}(x,y) f(y)dy,
\]
for any $x \in U$.
Moreover, $G_{U}$ is a fundamental solution for $\Op^{*}= \frac{1}{2} \sum_{i=1}^{m} X_{i}^{*}X_{i}^{*} + X_{0}^{*}$; ( $\Op^{*} G_{U}(x,\cdot) = \delta_{x}$ in the sense of distribution). Thus $G_{U}$ is smooth off the diagonal and the singularity at $(x,x)$ does not depend of the domain $U$ (indeed, if $U' \supset U$ $G_{U'}(x,\cdot)-G_{U}(x,\cdot)$ is $\Op^{*}$-harmonic in $U$ and, by hypoellipticity, smooth at $x$). Fix $U$.
Let now $(B^{1}, \dots, B^{m})$ be a m-dimensional Brownian motion and consider the solution $(x_{t})$  of the Stratonovich equation
\begin{align*}
dx_{t}= \sum_{i=1}^{m} X_{i}(x_{t})\circ dB^{i}_{t} + X_{0}(x_{t})dt, \ \ x_{0}=x, \label{SDE}
\end{align*}
killed at the first exit time $\tau$ from $U$, \[ \tau:= \inf \{t >0, \ x_{t} \in U^{c} \}\]. The Green function $G$ is also the density of occupation time measure of $(x_{t})_{0\leq t \leq \tau}$: for every test function $f \in C^{\infty}_{c}(U)$,
\[
\E \left[ \int_{0}^{\tau} f(x_{t}) dt \right ] = \int_{U} f(y) G_{U}(x,y) dy. 
\] 

Under some assumptions concerning the geometry induced by the Lie brackets of the vector fields we prove the following result, which generalize to a general hypoelliptic diffusion a result of Ben Arous and Gradinaru (\cite{B-A_Grad}) proved under a stronger geometric assumption. 
\begin{thm}\label{thm1}
Given $x\in \mathcal{M}$ we have 
\[
\lim_{\eps \to 0} \sup_{\vert y \vert_{x} < \eps} \left \vert G_{U}(x,y) \vert y \vert_{x}^{Q(x)-2} - \frac{1}{J_{x}(0)}g^{(x)}\bigl (0, \theta_{x}(y) \bigr ) \right \vert=0 .
\]
\end{thm} 
The precise notations are explained later and the proof, which follows the pattern of Ben Arous and Gradinaru's work, is done in the appendix. Nevertheless, let us briefly present the underlying idea. The way how the $\{X_{i}\}_{i=0, \dots, m}$ generate $T_{x}\M$  yields a family of dilatations of $T_{x}\M$, denoted by $T_{\lambda}$ for $\lambda >0$. Then we consider the rescaled diffusions in $T_{x}\M \simeq \R^{n}$ defined by $v^{(x,\eps)}_{t}:= T_{1/\eps}(x_{\eps^{2}t})$. Then, using the stochastic Taylor formula of Castell \cite{Castell93}, we write the  Taylor expansion $v^{(x,\eps)}_{t}= u^{(x)}_{t}+ \eps R(\eps,t)$ where the remainder term $R(\eps,t)$ is bounded in probability as $\eps$ goes to $0$. We show that the \emph{tangent process} $u^{(x)}_{t}$ which appears in the first term of this Taylor expansion has a smooth density of occupation time measure, say $g^{(x)}$. Using estimates obtained by Nagel, Stein and Wainger in \cite{N-S-W} we prove the convergence of the Green functions $G_{U}^{(x,\eps)}$ of the rescaled diffusions $v^{(x,\eps)}$ towards $g^{(x)}$; Theorem \ref{thm1} follows. 

\begin{rmq}
In the sub-elliptic situation described in Ben Arous and Gradinaru's work, the Green function $g^{(x)}$ of the tangent process is strictly positive and, in this case, Theorem 1 provides a precise description of the singularity. In a parabolic situation, where the drift $X_{0}$ is needed to generate a Lie algebra of full rank, $g^{(x)}(0, \cdot)$ may be null in some domain. For a $y$ which tends to $x$ such that $\theta_{x}(y)$ is in this domain, Theorem 1 do not give the precise behavior of $G(x,y)$. 
\end{rmq}

As we shall exclusively work in a fixed bounded neighbourhood $U$ of $x$ we shall write $G$ for $G_{U}$, and without loss of generality we shall suppose the vector fields $X_{0},X_{1}, \dots, X_{m}$ globally bounded with bounded derivatives and there is no explosion of the solutions to the SDE considered below.

Now let introduce the notations of the theorem and the tangent process. 

\subsection{Triangular basis}

For every multi-index $J=(j_{1}, \dots, j_{l}) \in \{ 0, 1, \dots, m \}^{l}$ we denote by $X^{J}$ the vector field
\[
X^{J}:= [ X_{j_{1}}, [  X_{j_{2}}, [ \cdots  [ X_{j_{l-1}}, X_{j_{l}}]] \cdots ] ] ,
\]
and if $J=(j)$, $X^{J}:=X_{j}$.

Let denote by $\vert J \vert$ the length $l$ of $J$ and by $\Vert J \Vert$ the \emph{order} of $J$:
\[
\Vert J \Vert:= \vert J \vert + \textrm{number of} \ 0 \ \textrm{in} \ J.
\] 

For $x \in \mathcal{M}$ set
\[
C_{i}(x) := \mathrm{Vect} \left \{ X^{J}(x); \quad \Vert J \Vert \leq i  \right \}.
\]
By assumption \eqref{H} we can consider the smallest integer $r(x)$ such that $C_{r(x)}= T_{x} \mathcal{M}$.
We denote by $Q(x)$ \emph{the graded dimension} of $Lie(X_{0}, \dots, X_{m})$ at $x$:
\begin{align*}
 Q(x):= \sum_{i=1}^{r(x)} i\times(\dim C_{i}(x)-\dim C_{i-1}(x)).
\end{align*}

\begin{Ass}\label{geomconst}
We assume that the geometry of the Lie brackets is constant, this means that the $\dim C_{i}(x)$ are constant, thus $r$ and $Q$ are also constant.
\end{Ass}

Let $\B=(J_{1},\dots,J_{n})$ be a family of multi-indexes such that $(X^{J}(x))_{J\in \B}$ is \emph{a triangular basis} of $T_{x}\M$, that is, for $j\leq r$, $\{X^{J}(x); \ J \in \B, \ \Vert J \Vert \leq j\}$ is a basis of  $C_{j}(x)$. By the previous assumption we have also $(X^{J}(y))_{J\in \B}$ is a triangular basis of $T_{y}\mathcal{M}$ for $y$ in a neighbourhood of $x$. We suppose that the domain $U$ is in a such neighbourhood.

For any multi-index $L$ there exist smooth realvalued functions $a_{J}^{L}$ on $U$ such that we have
\[
X^{L}= \sum_{J \in \B} a_{J}^{L} X^{J}.
\]
Since $(X^{J})_{J \in \B}$ is a triangular basis, if $\Vert J\Vert > \Vert L \Vert$ then $a_{J}^{L}= 0$ in $U$. 

\begin{rmq}
We have chosen one particular n-tuplet $\B$ and this choice seems to be not intrinsic. Nevertheless, the following section we show that under the assumption \ref{geomconst} this choice does not affect the results. 
\end{rmq}

\subsection{Homogeneous norm and a priori estimates}

Given a vector field $X$ on $\mathcal{M}$, $\exp(X)(x)$ denote the solution at time $1$ of the differential equation:
\[ \left \{ \begin{aligned} \frac{du}{ds}&= X(u(s)) \\ u(0)&=x. \end{aligned} \right. \]

By a well-known result about the dependence of the solution of a differential equation with respect to its parameters, the map 
\[
\varphi_{x}: \ u \in \R^{n} \mapsto \exp \left ( \sum_{i=1}^{n} u_{i} X^{J_{i}} \right )(x) \in \M
\]
is smooth, and since $(X^{J}(x))_{J \in \B}$ is a basis we can find a neighbourhood $W$ of $0$ in $\R^{n}$ such that $\varphi_{x}$ is a smooth diffeomorphism from $W$ onto $\varphi_{x}(W)$. There is no loss of generality in supposing that $U \subset \varphi_{x}(W)$. We define the homogeneous norm of $y= \varphi_{x}(u) \in U$ setting
\[
\vert y \vert_{x}:= \left [ \sum_{k=1}^{r} \left (\sum_{i, \Vert J_{i} \Vert=k} u_{i}^{2} \right )^{\frac{Q}{2k}} \right ]^{\frac{1}{Q}}.
\]

During the proof of Theorem \ref{thm1} we will use homogenous norm at different points near of the reference point $x$. For this we need assumption \ref{geomconst} and, since $\varphi_{x}$ depends smoothly  on $x$ we can take $U$ small enough so that every map $\varphi_{y}^{-1}:U\to \varphi_{y}^{-1}(U)$, with $y\in U$, is a diffeomorphism (see Corollary 4.1 of \cite{N-S-W}). Thus, $\vert z\vert_{y}$ is well defined for $z, y \in U$. 

Also, for $u \in \R^{n}$, set $\vert u \vert_{n}:= \left [ \sum_{k=1}^{r} \left (\sum_{i, \Vert J_{i} \Vert=k} u_{i}^{2} \right )^{\frac{Q}{2k}} \right ]^{\frac{1}{Q}}$, and $\Vert u \Vert$ is the Euclidean norm of $u$. 

Now we introduce the \emph{dilatations} on $\R^{n}$ associated to the family of multi-index $\B$. For $\varepsilon >0$, set
\[
\begin{matrix}
T_{\eps}:&\R^{n} & \longrightarrow& \R^{n} \\
&u & \longmapsto & (\eps^{\Vert J_{i} \Vert} u_{i} )_{i=1,\dots, n}
\end{matrix}
\] 
The norm $\vert \cdot \vert_{n}$ is homogeneous with respect to the dilatations $T_{\eps}$; as  $\vert T_{\eps} u \vert_{n}= \eps \vert u \vert_{n}$.

For $y \in U$, denote by $\theta_{x}(y)$ the \emph{angular variable} ($ \vert \theta_{x}(y) \vert_{n}= 1$)
\[
\theta_{x}(y):= T_{1/ \vert y \vert_{x}} \circ \varphi_{x}^{-1} (y).
\]
\begin{rmq}
We can suppose $\B$ indexed in such a way that $(\Vert J_{i} \Vert)_{i=1,\dots,n}$ increase. If $\B'=(J'_{i})_{i=1,\dots, n}$ is an other family with this property, then $\Vert J_{i} \Vert =\Vert J'_{i} \Vert$ and the definition of the dilatations do not depend on the particular family $\B$ chosen. 
\end{rmq}

Let now introduce the results obtained by Nagel, Stein and Wainger in \cite{N-S-W} giving estimates on the Green function in terms of a geometric pseudo-metric $\rho$ associated with the vector fields $X_{0}, X_{1}, \dots, X_{m}$.

For $\delta >0$ denote by $C(\delta)$ the class of absolutely continuous maps $\phi: [0,1] \to U$ such that we can write 
\[
\phi'(t) = \sum_{ \Vert J \Vert \leq r} a_{J}(t) X^{J}(\phi(t))
\]
for almost all $t \in [0,1]$ with functions $a_{J}: [0,1] \to U$ such that $\vert a_{J}(t) \vert < \delta^{\Vert J \Vert}$. 

We define a pseudo-metric $\rho$ on $U$ setting 
\[
\rho(y,z)= \inf \left \{ \delta >0 \vert \quad \exists \phi \in C(\delta) \ \mathrm{such \ that} \ \phi(0)=y, \ \phi(1)=z \right \}
\]
for all $y,z \in U$. 
By Corollary p 114 of \cite{N-S-W}, provided $U$ is  small enough, there exists a constant $C>0$ such that 

\begin{align}
\vert G(y, z) \vert \leq C \frac{\rho^{2}(y,z)}{\mathrm{Vol} \left (B(y,\rho(y,z))\right )} \label{estimapriori} \\
\left \vert X_{j_{1}}\cdots X_{j_{k}}G(y,z) \right \vert  \leq \frac{\rho^{2-\Vert J \Vert }(y,z)}{\mathrm{Vol} \left (B(y,\rho(y,z))\right )} \label{estimapriori2}
\end{align}
holds for any $y,z \in U$, and for any multi-index  $ J=(j_{1},\dots, j_{k})$.

\begin{rmq}
In the case where $X_{0}=0$, Theorem 4 of \cite{N-S-W} ensures that $\rho$ is locally equivalent to the sub-Riemanian distance associated to $\{X_{i} \}_{i=1, \dots ,m}$.
\end{rmq}

In the following proposition we show that the homogeneous norm and the pseudo-metric are locally equivalent. The proof of this result, which is a sort of ``ball-box theorem'' and requires the assumption \ref{geomconst}, is given in appendix.

\begin{prop}\label{loceq}
Provided $U$ is small enough, there exist positive constants $C_{1}, C_{2}$ such that 
\begin{align}
\forall y,z \in U, \quad C_{1}\rho(y,z) \leq \vert z\vert_{y} \leq C_{2}\rho(y,z).
\end{align}
\end{prop}
\begin{rmq}
This result shows in particular that if $\vert \cdot \vert'_{\cdot }$ is the family of homogeneous norm associated to a $n$-tuple $\B'$ such that $(X^{L})_{L \in \B'}$ is a triangular basis then $\vert \cdot \vert_{\cdot}$ is locally equivalent to $\vert \cdot \vert'_{\cdot}$. Thus, the particular choice of basis $\B$ does not matter. 
\end{rmq}

We deduce from the previous proposition and estimates \eqref{estimapriori} and \eqref{estimapriori2} of \cite{N-S-W} the following \emph{a priori} estimates. The proof is given in appendix.
\begin{prop}\label{propestimapriori}
There is a constant $C>0$ such that
\begin{align*}
\forall y \neq z \in U, \ \vert G(y,z) \vert \leq \frac{C}{\vert z \vert_{y}^{Q-2}}.
\end{align*}
Moreover for any integer $k$, for any multi-index $J=(j_{1},\dots, j_{k})\in \{0, \dots, m \}^k$, and for any $y\neq z$ in $U$ 
\begin{align*}
\vert X_{j_{1}} \cdots X_{j_{k}} G(y,z) \vert \leq \frac{C}{\vert z \vert^{Q-2+\Vert J \Vert }_{y}}.
\end{align*}
\end{prop}

Another consequence of the local equivalence between $\rho$ and $\vert \cdot \vert_{\cdot}$ is the following result needed in the proof of Theorem \ref{thm1}. It is a direct consequence of Proposition \ref{loceq} and Proposition (1.1) p 107 and (iii') p 109 in \cite{N-S-W}.
\begin{prop}[Triangular inequality and comparaison with a Euclidian norm]\label{prop4}
For any small enough compact set $U$, we can find a positive constant $c_{0}$ such that for any $ w,y,z \in U$:
\begin{align}
 \vert y \vert_{w} \leq c_{0}\left ( \vert z\vert_{w} + \vert z \vert_{y} \right ). \label{intri}\end{align}
Moreover, there exist positive constants $c',c''$ such that for any $y, z \in U$
\begin{align}
 c' \Vert \varphi^{-1}_{y}(z)\Vert \leq \vert z\vert_{y}\leq c'' \Vert \varphi^{-1}_{y}(z) \Vert^{1/r} \label{compnormeucli}.
 \end{align}
\end{prop}
\subsection{Taylor expansion and tangent process}

The underlying idea in Theorem \ref{thm1} is to use the dilations $T_{\eps}$ to zoom on the trajectories of the diffusion $x_{t}$. For this we introduce the \emph{rescaled diffusions} 
\[
 v^{(x,\eps)}_{t}:= T_{1/\eps} \circ \varphi_{x}^{-1} (x^{\eps}_{t}),
\]
where $x^{\eps}_{t}$ is a diffusion solution of 
\[
d x^{\eps}_{t} = \eps \sum_{i=1}^{m} X_{i} ( x^{\eps}_{t}) \circ d B^{i}_{t} + \eps^{2} X_{0}(x^{\eps}_{t}) dt, \quad x^{\eps}_{0}=x,
\]
and has the same law as $x_{\eps^{2}t}$. 
These are diffusions with values in the neighbourhood $\tilde{U}^{\eps}:= T_{1/\eps} \circ \varphi_{x}^{-1}(U)$ of $0 \in \R^{n}$, defined up to the exit time $\tau_{\eps}:= \inf\{t>0, v^{(x,\eps)}_{t} \notin \tilde{U}^{\eps} \}$ of $\tilde{U}^{\eps}$. Note that $\tau_{\eps}$ has the same law as $\tau/ \eps^{2}$ where $\tau$ is the exit time for $x_{t}$ from $U$. In order to write a Taylor expansion of $v_{t}^{(x,\eps)}$ for $\eps$ near $0$, we introduce notations taken from \cite{Castell93}. For a multi-index $J=(j_{1}, \dots, j_{l}) \in \{0,\dots, m  \}^{l}$ we denote by $B^{J}_{t}$ the Stratonovitch iterated integral
\[
B_{t}^{J}:=\int_{\Delta_{t}^{l}} \circ dB_{t_{1}}^{j_{1}}\dots \circ dB_{t_{l}}^{j_{l}}, 
\]
where $\Delta_{t}^{l}=\{(t_{1}, \dots, t_{l}); \ 0 <t_{1}< \dots < t_{l}< t \}$ and $B^{0}_{t}=t$.
For $\sigma$ a permutation of $\{1, \dots, l \}$, we set $J\circ \sigma= (j_{\sigma(1)}, \dots, j_{\sigma(l)})$ and denote by $e(\sigma)= \mathrm{Card}\{ j \in \{1, \dots, l-1\}; \sigma(j) > \sigma(j+1) \}$ the number of errors in ordering $\sigma$. Then, we denote by \emph{$c_{t}^{J}$ the linear combination of Stratonovich iterated integrals}
\[
c_{t}^{J}:= \sum_{\sigma \in \mathfrak{S}_{\vert J \vert}}\frac{(-1)^{e(\sigma)}}{\vert J\vert^{2} C_{\vert J \vert -1}^{e(\sigma)}}B_{t}^{J\circ \sigma^{-1}}.
\] 
Now let introduce the \emph{tangent process} $u_{t}^{(x)}$, which is a $\R^{n}$-valued process defined in terms of the Stratonovich integrals by:
\[
u_{t}^{(x)}:= \left ( \sum_{L, \Vert L \Vert = \Vert J_{i} \Vert} a^{L}_{J_{i}}(x) c_{t}^{L} \right)_{i=1,\dots,n}.
\]
This process is not intrinsic and depends on the choice of the $n$-tuple $\B$. In general it is not a diffusion but we will show that it is the projection in $\R^{n}$ of a diffusion in some bigger space. It is the first term in the Taylor expansion of $v^{(x,\eps)}_{t}$ as $\eps$ goes  to $0$. More precisely, using the results of Castell \cite{Castell93} we obtain the following proposition proved in appendix. 
\begin{prop}\label{taylor}
Set
\[
T^{\eps}_{U}:= \inf \left \{ t > 0, \quad \exp \left ( \sum_{L, \Vert L \Vert \leq r} \eps^{\Vert L \Vert} c_{t}^{L} X^{L}\right )(x) \notin U \right \}.
\]

For all $\eps >0$ and $t \leq \tau_{\eps} \wedge T^{\eps}_{U} $ we define $R^{(x)}(\eps,t)$ by the formula
\[
v^{(x,\eps)}_{t} = u_{t}^{(x)} + \eps R^{(x)}(\eps,t).
\]
Then $R^{(x)}(\eps, t)$ is bounded in probability; more precisely for $T>0$ fixed, there exist positive constants $\alpha $ and $c$ such that 
\[ 
\forall R > c, \quad 
\lim_{\eps \to 0} \Prob \left ( \sup_{0\leq s\leq T} \Vert R^{(x)}(\eps,s)  \Vert >R, \ T < \tau_{\eps} \wedge T_{U}^{\eps}  \right ) \leq \exp \left( -\frac{R^{\alpha}}{cT}\right ).
\]
\end{prop} 
We prove in proposition \ref{tangent} below that $u^{(x)}_{t}$ admits smooth Green function $g^{(x)}$; the first step to do that is to show that it is the projection of a diffusion taking values on a universal Lie group.

\begin{itemize}
\item Denote by $\mathcal{G}_{(m,r)}$ the formal r-step free Lie algebra generated by $(Y_{0}, Y_{1}, \dots , Y_{m})$ such that if $J$ is a multi-index with $\Vert J \Vert > r$ then $Y^{J}=0$. This r-step free algebra can be obtained by taking the quotient of the free Lie algebra generated by $(Y_{0},\dots, Y_{m})$ by the ideal $\mathrm{Vect} \{ Y^{J}, \Vert J \Vert \geq r+1 \}$. We denote by $G_{(m,r)}$ the nilpotent, connected and simply connected, Lie group with Lie algebra $\mathcal{G}_{(m,r)}$. Denoting by $\tilde{Y}_{i}$ the left-invariant vectors fields associated to the $Y_{i}$, we consider the $G_{(m,r)}$-valued diffusion $\tilde{x}_{t}$ solution, starting at $e$, of the stochastic differential equation
\[
d \tilde{x}_{t} = \sum_{i=0}^{d} \tilde{Y}_{i}(\tilde{x}_{t}) \circ dB^{i}_{t},
\]
with $B^{0}_{t}=t$.

In \cite{Castell93} Castell shows the Chen-Strichartz formula
\begin{align}
\tilde{x}_{t} = \exp \left ( \sum_{k=1}^{r} \sum_{\Vert L \Vert = k} c_{t}^{L} \tilde{Y}^{L}\right )(e). \label{CS}
\end{align}
Setting $V_{k}= \mathrm{Vect} \{ Y^{J}, \Vert J \Vert =k \}$ we have the direct sum decomposition
\[
\mathcal{G}_{(m,r)}= V_{1} \oplus \cdots \oplus V_{r}.
\]
\item We can complete $\mathcal{B}$ in $\mathcal{A}$ such that  for $k=1, \dots, r$, $(Y^{K})_{K \in \mathcal{A}, \Vert K \Vert =k}$ is a basis of $V_{k}$. 
 There are constants $(b_{K}^{L})$ such that for any multi-index $L$
\[
Y^{L}= \sum_{K \in \mathcal{A}} b^{L}_{K} Y^{K},
\]
with $b_{K}^{L}=0$ if $\Vert K \Vert \neq \Vert L \Vert$. The Lie algebra $\mathcal{G}_{(m,r)}$ is free up to order $r$ so for multi-indices $L$ of order smaller than $r$ these previous equalities are universal and hold for any family $(Y_{i})_{i=0\dots m}$ in any Lie algebra. In particular for $L$ with $\Vert L \Vert \leq r$ we have
\begin{align*}
X^{L} = \sum_{ K \in \mathcal{A}} b^{L}_{K} X^{K}= \sum_{K \in \mathcal{A}} b_{K}^{L} \left ( \sum_{J \in \mathcal{B}} a_{J}^{K} X^{J}   \right ) = \sum_{J \in \mathcal{B}} \left ( \sum_{K \in \mathcal{A}} b_{K}^{L} a_{J}^{K}\right ) X^{J},
\end{align*}
and so $a_{J}^{L}= \sum_{K\in \mathcal{A}} b_{K}^{L} a_{J}^{K} $ and if $\Vert L \Vert = \Vert J \Vert $ then 
\begin{align}
a_{J}^{L}= \sum_{\underset{\Vert K \Vert = \Vert J \Vert }{K \in \mathcal{A}}} b_{K}^{L}a_{J}^{K}. \label{trucdutsu}
\end{align}
We rewrite $\eqref{CS}$
\[
\tilde{x}_{t}= \exp \left ( \sum_{K \in \mathcal{A}} \left (  \sum_{L, \Vert L \Vert = \Vert K \Vert } b^{L}_{K} c_{t}^{L}\right )  \tilde{Y}^{K} \right )(e)
\]
and we denote by $\tilde{u}_{t}^{(e)}$ the process 
\[
\tilde{u}_{t}^{(e)}:= \left ( \sum_{L, \Vert L \Vert = \Vert K \Vert } b^{L}_{K} c_{t}^{L}\right )_{ K \in \mathcal{A}}.
\]
Set $d:= \mathrm{Card} \mathcal{A} - \mathrm{Card} \mathcal{B}= \mathrm{dim} \mathcal{G}_{(m,r)} -n$, and introduce the projection $p_{x}: \R^{d+n} \to \R^{n}$ defined by: 
\[
p_{x}(\tilde{u}):= \left (\sum_{ \underset{\Vert K \Vert = \Vert J\Vert}{ K \in \mathcal{A}}} a_{J}^{K}(x) \tilde{u}_{K} \right )_{J \in \mathcal{B}}.
\]
We obtain
\begin{align*}
p_{x}(\tilde{u}_{t}^{(e)})&= \left ( \sum_{ \underset{\Vert K \Vert = \Vert J \Vert}{K \in \mathcal{A}}} a_{J}^{K}(x) \left ( \sum_{\underset{\Vert L \Vert = \Vert K\Vert}{L}} c_{t}^{L}b_{K}^{L}\right) \right )_{J \in \mathcal{B}} 
= \left ( \sum_{ \underset{\Vert L \Vert=\Vert J \Vert }{L}} c_{t}^{L} \underset{= a_{J}^{L}  \ \mathrm{by} \  \eqref{trucdutsu}}{\underbrace{\left ( \sum_{ \underset{\Vert K \Vert =\Vert J \Vert  }{K \in \mathcal{A}}} a_{J}^{K}(x)b_{K}^{L}\right)}} \right )_{J \in \mathcal{B}}
= u_{t}^{(x)}.
\end{align*}
\end{itemize}
Now we prove 
\begin{prop}\label{tangent}
The tangent process $u_{t}^{(x)}$ admits a smooth Green function, denoted by $g^{(x)}(0, \cdot )$. 
\end{prop}
\begin{proof}
Since $G_{(m,r)}$ is nilpotent, connected and simply connected, the map 
\[
\tilde{\psi}_{e}: \tilde{u} \mapsto \tilde{\psi}_{e}(\tilde{u})= \exp \left( \sum_{K \in \mathcal{A}} \tilde{u}_{K}\tilde{Y}^{K} \right)(e)
\]
is a diffeomorphism of $\mathcal{G}_{(m,r)}$ onto $G_{(m,r)}$. Thus the process $u_{t}^{(e)}= \tilde{\psi}_{e}^{-1}( \tilde{x}_{t} )$ is a hypoelliptic $\mathcal{G}_{(m,r)}$-valued diffusion and we denote by $\tilde{g}$ its Green function. We identify $\mathcal{G}_{(m,r)}$ to $\R^{n+d}$ via the basis $(X^{K})_{K \in \mathcal{A}}$.

For $\tilde{u} \in \R^{n+d}$ we write $\tilde{u}=(u,v)$ where $u:= (\tilde{u}_{K})_{K \in \mathcal{B}} \in \R^{n}$ and $v:= (\tilde{u}_{K})_{K \in \mathcal{A} \setminus \mathcal{B}} \in \R^{d}$ and we note that
\[
p_{x}(\tilde{u})=p_{x}((u,v))= u + M_{x} (v),
\]
where $M_{x}$ is the $n \times d$-matrix defined by
\[
M_{x}(v) = \left ( \sum_{\underset{\Vert K \Vert = \Vert J \Vert}{K \in \mathcal{A} \setminus \mathcal{B}} } a_{J}^{K} v_{K}  \right )_{J \in \mathcal{B}}. 
\]
Now for a test function $\varphi \in C_{c}^{\infty}(\R^{n})$ we obtain
\begin{align*}
\mathbb{E} \left [ \int_{0}^{+ \infty} \varphi(u_{t}^{(x)}) dt\right ] = \mathbb{E} \left [  \int_{0}^{+\infty} \varphi(p_{x}(\tilde{u}_{t}^{(e)})) dt\right ] &= \int_{\R^{n}\times \R^{d}} \varphi(u +M_{x}(v)) \tilde{g}(0,(u,v)) dudv \\ 
&= \int_{\R^{n}} \varphi(u) \left ( \int_{\R^{d}} \tilde{g}\left (0, \left (u-M_{x}(v), v \right ) \right )dv \right ) du.
\end{align*}  
Then $u_{t}^{(x)}$ admits a Green function denoted by $g^{(x)}$, which verifies:
\begin{align}
g^{(x)}(0,u)= \int_{\R^{d}} \tilde{g}\left (0, \left (u-M_{x}(v), v \right ) \right )dv. \label{gtilde}
\end{align}
In the proof of Fact 2 of Proposition \ref{prop1} in appendix we dominate the integrand by a integrable term. Thus $g^{(x)}(0,u)$ is finite for $u\neq 0$ and the dominated convergence theorem ensures that it is smooth as a function of $u$.  
\end{proof}
Since $u_{\eps^{2}t}^{(x)}$ has the same law as $T_{\eps}(u_{t}^{(x)})$ we deduce that (recall that $v=T_{\eps}(u) \Rightarrow dv = \eps^{Q}du$)
\begin{align}
g^{(x)}\left (0, T_{1/\varepsilon}(u)\right )= \varepsilon^{Q-2}g^{(x)}(0,u). \label{rescalingpt}
\end{align}
For $\eps= 1/ \vert y \vert_{x}$ we obtain
\begin{align} 
g^{(x)}(0,\varphi_{x}^{-1}(y))= \frac{1}{\vert y \vert_{x}^{Q-2}}g^{(x)}(0, \theta_{x}(y)).  \label{chouette}
\end{align}

Denote by $G^{(x, \eps)}$ the Green function of the \emph{rescaled diffusions} $v^{(x, \eps)}_{t}$ on $\tilde{U}^{\eps}= T_{1/\eps} \circ \varphi_{x}^{-1}(U)$. The following computation links $G^{(x, \eps)}$ to $G$.

Denoting by $J_{x}:= \vert \mathrm{Jac}(\varphi_{x}) \vert$ the Jacobian of $\varphi_{x}$ on $\tilde{U}^{\eps}$ we obtain for $u\in \tilde{U}^{\eps}$ and $\psi\in C_{c}^{0}(\tilde{U}^{\eps}\setminus \{u\})$
\begin{align*}
\int_{\tilde{U}^{\eps}}\psi(v)G^{(x,\eps)}(u,v)dv&= \mathbb{E}_{u} \left [ \int_{0}^{\tau_{\eps}} \psi(v^{(x,\eps)}_{t})dt \right ] \\
					&=\mathbb{E}_{\varphi_{x}\circ T_{\eps}(u)} \left [ \int_{0}^{\tau_{\eps}} \psi\left (T_{1/ \eps}\circ \varphi_{x}^{-1}(x_{\eps^{2}t}) \right )dt \right ] \\
					&= \mathbb{E}_{\varphi_{x}\circ T_{\eps}(u)} \left [ \eps^{-2} \int_{0}^{\tau} \psi \left (T_{1/ \eps}\circ \varphi_{x}^{-1}(x_{t})\right )dt \right ] \\
					&=\eps^{-2}\int_{U} \psi \left (T_{1/ \eps}\circ \varphi_{x}^{-1}(y)\right )G(\varphi_{x}\circ T_{\eps}(u),y)dy \\
					&= \eps^{Q-2} \int_{\tilde{U}^{\eps}} \psi(v)G\left (\varphi_{x}\circ T_{\eps}(u), \varphi_{x}\circ T_{\eps} (v)\right) J_{x}(T_{\eps}(v)) dv.
\end{align*}

Thus
\begin{align}
 G^{(x,\eps)}(u,v)=\eps^{Q-2}J_{x}(T_{\eps}(v)) G(\varphi_{x}\circ T_{\eps}(u),\varphi_{x}\circ T_{\eps}(v)), \label{zib}
\end{align}
and in particular, for $y\in U$ 
\begin{align}
G(x,y)= \frac{1}{J_{x}( \varphi_{x}^{-1}(y))\eps^{Q-2}}G^{(x,\eps)}(0, T_{1/\eps}\circ \varphi_{x}^{-1}(y)).  \label{bidule}
\end{align}

Let us now assume the following assumption about the dimensions of the spaces $C_{i}$.
\begin{Ass}\label{ass2}
We assume
\begin{enumerate}
\item $r\geq2$
\item $\dim C_{i}-\dim C_{i-1} \geq 1, \   \forall i=2,\dots,r$
\item $\dim C_{1} \geq2.$
 \end{enumerate}
\end{Ass}

Under these assumptions, and using Taylor expansion of $v_{t}^{(x,\eps)}$ as described in Proposition \ref{taylor},  we show the convergence of $G^{(x, \eps)}$ to $g^{(x)}$, uniformly on compact sets of $\R^{n} \setminus \{0\}$.
\begin{prop}\label{prop1}
For any compact sets $K \subset \R^{n}\setminus \{ 0\}$, we have:
\[
\sup_{u\in K}\left \vert G^{(x,\eps)}(0,u)-g^{(x)}(0,u) \right \vert \underset{\eps \to 0}{\longrightarrow} 0.
\]
\end{prop}

Considering \eqref{chouette} and \eqref{bidule}, taking for $K$, the unit sphere of $\R^{n}$ for the homogeneous norm $\vert \cdot \vert_{n}$, and $\eps= \vert y \vert_{x}$, Theorem \ref{thm1}  follows from Proposition \ref{prop1}.

As an illustration of the interest of that, we apply it in the next section to a diffusion in the Poincar\'e group constructed as a lift of the relativistic diffusion defined by Dudley in \cite{Dud66}; this leads to a Wiener criterion of thinness and a Poincar\'e cone condition. 

\section{A Wiener criterion and a Poincar\'e cone condition in the Poincar\'e group} \label{sec3}

\subsection{Geometric framework.}

Denote by $(\xi^{0}, \dots, \xi^{d} )$ the coordinates of a point $\xi$ with respect to the canonical basis $(e_{0}, e_{1}, \dots, e_{d})$ of $\R^{d+1}$ and set $\R^{1,d}$ the space $\R^{d+1}$ endowed with the Minkowski's quadratic form $q$
\begin{align*}
q(\xi)=\left ( \xi^{0} \right )^{2}-\sum_{i=1}^{d} \left(\xi^{i} \right)^{2}.
\end{align*}
Denote by $SO(1,d)$ the sub-group of  $SL(\R^{d+1})$ made up of direct $q$-isometries, and by $SO_{0}(1,d)$ the connected component of the identity in $SO(1,d)$. The Poincar\'e group is the group $G:=SO_{0}(1,d) \ltimes \R^{1,d}$ of affine isometries with group law defined by
\begin{align*}
(g,\xi)(g',\xi')=(gg', \xi +g\xi').
\end{align*}
We can see $G$ as the following matrix sub-group of $SL(d+2)$
\begin{align*}
G=\left \{\left(\begin{array}{cc}g & \xi \\0 & 1\end{array}\right); \ g \in SO_{0}(1,d), \ \xi \in \R^{(1,d)}\right \}.
\end{align*}
The Lie algebra of $SO(1,d)$ is
\begin{align*}
so(1,d)=\left \{ \left(\begin{array}{cc}0 & {}^t{(u_{i})} \\(u_{i}) & (u_{ij})\end{array}\right); \ (u_{i})\in \R^{d},(u_{ij})\in \R^{d\times d}; \ (u_{ji})=-(u_{ij}) \right \}.
\end{align*}
Thus, $SO_{0}(1,d)=\exp( so_{(1,d)})$ and $G=\exp(\mathfrak{g})$ where
\begin{align*}
\mathfrak{g}= \left \{ \left(\begin{array}{ccc}0 &  {}^t{(u_{i})} & u_{0} \\(u_{i}) & (u_{ij}) & (u_{i0}) \\0 & 0 & 0 \end{array}\right); (u_{i})_{i=1,\dots, d} \in \R^{d}, \ (u_{ji})=-(u_{ij}), \ (u_{0}, (u_{i0})_{i=1,\dots, d})  \in \R^{1,d} \right \}.
\end{align*} 
Denote by $\Hyp^{d}$ the half-unit sphere of $\R^{1,d}$
\[
\Hyp^{d}:= \left \{ \xi \in \R^{1,d} \vert \quad q(\xi)=1 \ \mathrm{and} \ \xi^{0}>0   \right \}.
\]
Endowing $T\Hyp^{d}$ with the metric $-q\vert_{T\Hyp^{d}}$ turns $\Hyp^{d}$ into a Riemaniann manifold of constant negative curvature. This is the hyperboloid model for the hyperbolic space. 

\subsection{Relativistic diffusion.}
We introduce a left invariant diffusion on $G$ which is a lift of the relativistic diffusion on $\Hyp^{d} \times \R^{1,d}$ introduced by Dudley in \cite{Dud66}. The asymptotic behavior of this $G$-valued diffusion was studied in \cite{Bailleul-Raugi}. A relativistic diffusion can be seen as a stochastic perturbation of the geodesic flow on a Lorentzian manifold. For more information on this subject see \cite{Bailleul2010}.

For  $i=1\dots d$, denote by $X_{i}$ the left invariant vector field on  $G$ defined by
\begin{align*}
X_{i}(g,\xi)=(g,\xi)\underset{:=E_{i}\in \mathfrak{g}}{\underbrace{\left(\begin{array}{ccc}0 & {}^t{e_{i}}&0 \\e_i & \mathbf{0} &0\\0 & 0&0\end{array}\right)}}.
\end{align*}
We denote by $X_{0}$ the left invariant vector field on  $G$ defined by
\begin{align*}
X_{0}(g,\xi)=(g,\xi)\underset{:=E_{0}\in \mathfrak{g}}{\underbrace{\left(\begin{array}{ccc}0 & 0&1 \\0 & \mathbf{0} &0\\0 & 0&0\end{array}\right)}}.
\end{align*}

We denote by $\{ (g_{s},\xi_{s}) \}_{s\geq 0}$ the diffusion on $G$ solving
\begin{align*}
d(g_{s},\xi_{s})= \sum_{i=1}^{d} X_{i}(g_{s},\xi_{s})\circ dB^{i}_{s}+X_{0}(g_{s},\xi_{s})ds.
\end{align*}
it has generator
\begin{align*}
\Op= \frac{1}{2}\sum_{i=1}^{d} X_{i}^{2}+ X_{0}.
\end{align*}

Note that $(g_{t}(e_{0}), \xi_{t})$ takes values in $\Hyp^{d}\times \R^{1,d}$. We have the following proposition (cf \cite{Bailleul2010}, \cite{Bailleul-Raugi}, \cite{F-LJ})
\begin{prop}
The process $(g_{t}(e_{0}),\xi_{t})$ in $\Hyp^{d} \times \R^{1,d}$ is Dudley's relativistic diffusion and has generator
\[
\frac{1}{2}\Delta_{\Hyp^{d}} + H_{0}
\]
where the vector field $H_{0}$ generates the geodesic flow in $T^{1}\R^{1,d}\equiv \Hyp^{d}\times \R^{1,d}$. In other words $g_{t}(e_{0})$ is a Brownian motion in $\Hyp^{d}$ and $\xi_{t}$ is its time integral: $\xi_{t}=\int^{t}_{0}g_{s}(e_{0}) ds$.
\end{prop}
A simple computation shows that, for $i,j=1\dots d$ 
\begin{align*}
[X_{i},X_{j}](g,\xi)&=(g,\xi)\underset{:=E_{ij}\in \mathfrak{g}}{\underbrace{\left(\begin{array}{ccc}0 & 0 & 0 \\0 & e_i \otimes e_j- e_j\otimes e_i & 0 \\0 & 0 & 0  \end{array}\right)}}\\
[X_{i},X_{0}](g,\xi)&=(g,\xi)\underset{:=E_{i0}\in \mathfrak{g}}{\underbrace{\left(\begin{array}{ccc}0 & 0 & 0 \\0 &\mathbf{0}& e_{i} \\0 & 0 & 0  \end{array}\right)}}\\
\end{align*}
Thus at every point $(g,\xi)$ of $G$ we have
\[ \mathrm{Vect}\{X_{0}, X_{i},  [X_{i},X_{j}], [X_{i},X_{0}],\  i,j=1\dots d \}=T_{(g,\xi)}G,
\]
so $\Op$ is hypoelliptic by H\"ormander's Theorem.Moreover we have
\begin{align}
[X_{i},[X_{i},X_{0}]]=X_{0},
\end{align}
and then, 
\[ \mathrm{Vect}\{X_{i},  [X_{i},X_{j}], [X_{i},X_{0}],[X_{i},[X_{i},X_{0}]]\  i,j=1\dots d \}=T_{(g,\xi)}G.
\]
By H\"ormander's Theorem $g_{t}$ has a smooth density with respect to the $Haar$ measure of $G$.
With the notation of  Section \ref{sec2} the vectors fields $X_{i}$ induce the following graduation of $T_{(g,\xi)}G= (g,\xi)\mathfrak{g}$:
\begin{align}
C_{1}(g,\xi)&= \left \{ (g,\xi)\left(\begin{array}{ccc}0 & {}^t{(u_{i})} & 0 \\(u_{i}) & 0 & 0 \\0 & 0 & 0\end{array}\right); (u_{i}) \in \R^{d}  \right \} \\
C_{2}(g,\xi)&= \left \{ (g,\xi)\left(\begin{array}{ccc}0 & {}^t{(u_{i})} & u_{0} \\(u_{i}) & (u_{ij}) & 0 \\0 & 0 & 0\end{array}\right); u_{0}\in \R, (u_{i})\in \R^{d}, u_{ji}=-u_{ji}   \right \} \\
C_{3}(g,\xi)&=T_{(g,\xi)}G.
\end{align}
Thus we have $r(g,\xi)=3$.
We choose the triangular basis $\B:=(X_{i},X_{0}, [X_{i},X_{j}], [X_{i},X_{0}])_{i<j=1\dots d}$  for $T_{(g,\xi)}G=(g,\xi)\mathfrak{g}$.
The graded dimension $Q$ is constant on $G$ and is computed explicitly
\[
Q(g,\xi)= d+ 2(d(d-1)/2 +1) + 3d=d^{2}+3d+2.
\]
Note that assumptions \ref{geomconst} and \ref{ass2} are fulfilled.

The family of dilations on $\mathfrak{g}$ is
\begin{align}
T_{\eps}\left(\begin{array}{ccc}0 &  {}^{t}{(u_{i})} & u_{0} \\(u_{i}) & (u_{ij}) & (u_{i0}) \\0 & 0 & 0 \end{array}\right)=\left(\begin{array}{ccc}0 &  \eps \ {}^t{(u_{i})} & \eps^{2}u_{0} \\ \eps (u_{i}) & \eps^{2}(u_{ij}) & \eps^{3}(u_{i0}) \\0 & 0 & 0 \end{array}\right)
\end{align}

The homogeneous norm is given at $\mathbf{e}=(id,0)$ by the formula
\begin{align}
\left \vert \exp\left(\left(\begin{array}{ccc}0 &  {}^{t}{(u_{i})} & u_{0} \\(u_{i}) & (u_{ij}) & (u_{i0}) \\0 & 0 & 0 \end{array}\right)\right ) \right \vert_{\mathbf{e}}= \left( (\sum_{i=1}^{d} u_{i}^{2})^{Q/2} + (u_{0}^{2}+ \sum_{1=i<j=d} u_{ij}^{2})^{Q/4}+ (\sum_{i=1}^{d} u_{i0}^{2})^{Q/6} \right)^{1/Q}
\end{align} 

The $\mathfrak{g}$-valued tangent process is explicitly given by
\begin{align}
u_{t}^{\mathbf{e}}=\left(\begin{array}{ccccc}0 & B_{t}^1 & \cdots & B_{t}^d & t \\B_{t}^1 &  &   &   & \frac{1}{2}\left ( \int_{0}^t B^{1}_s ds -\int_{0}^t s\circ dB_{s}^1\right) \\\vdots &   & \left(\frac{1}{2} \left( \int_{0}^{t}B^{i}_s\circ dB^{j}_s-\int_{0}^{t}B^{j}_s\circ dB^{i}_s\right) \right)_{i=1\dots d}^{j=1\dots d} &   & \vdots \\B_{t}^d &   &   &   & \frac{1}{2}\left ( \int_{0}^t B^{d}_s ds -\int_{0}^t s\circ dB_{s}^d\right) \\0 &   & 0 & 0 & 0\end{array}\right)
\end{align}

Proposition 8 of \cite{Baill06} establishes that the support of $\{(g_{t},\xi_{t})\}_{t\geq0}$ coincides with the future half cone at $(g_{0},\xi_{0})$.
Thus the Green function is strictly positive on this domain:
\[ G(\mathbf{e}, (g,\xi))> 0 \Longleftrightarrow q(\xi)>0  \ \text{and} \ \xi^{0}>0. \]

By scaling, the support of $T_{1/\eps}\circ \exp^{-1}(g_{\eps^{2}s},\xi_{\eps^{2}s})$ is the half cone
\[ \left \{ (g,\xi); \ (\eps^{2} \xi^{0})^{2}- \Vert \eps^{3} \vec{\xi} \Vert^{2} >0 , \xi^{0}>0 \right \}.\]
When $\eps$ goes to $0$ this half cone becomes eventually the half space $\{(g,\xi) \in  G \vert \xi^{0}>0 \}$. 
From Proposition \ref{prop1} we deduce that the  support of the tangent process coincides with this half space.

We call a \textbf{homogeneous cone} with vertex $(g_{0},\xi_{0})$, a subset $C_{h}(g_{0},\xi_{0})$ of $G$ which is invariant under the dilations  $T_{\eps}$ at $(g_{0},\xi_{0})$ such that the ``sole'' \[\left \{(g,\xi) \in C_{h}(g_{0},\xi_{0}); \ \vert (g,\xi) \vert_{(g_{0},\xi_{0})}=1\right \}\]
is a compact subset of the half space
\[
\left \{  (g,\xi) \in G; \ (\xi-\xi_{0})^{0} >0 \right \}.
\]
Since, for such a cone $C_{h}(g_{0}, \xi_{0})$, the sole is compact in a domain where the Green function $g^{(g_{0},\xi_{0})}$ of the tangent process is positive, we can find two positive constants $\alpha, \beta$ such that 
\[
\alpha \leq g^{(g_{0}, \xi_{0})}(0,\cdot) \leq \beta
\]
in the sole. 
By Theorem  \ref{thm1} we can find a neighbourhood $U$ of $(g_{0}, \xi_{0})$ such that for every set $B \subset U$ in the homogeneous cone $C_{h}(g_{0},\xi_{0})$ we have
\begin{align}
\forall (g,\xi) \in B, \ \ \frac{\alpha}{\vert (g,\xi) \vert_{(g_{0}, \xi_{0})}^{Q-2}} \leq G\left ((g_{0},\xi_{0}), (g,\xi) \right) \leq \frac{\beta}{\vert (g,\xi) \vert_{(g_{0}, \xi_{0})}^{Q-2}}. \label{doubleinegalite}
\end{align}
\textbf{Recall:} $Q=d(d+3)+2.$
\begin{rmq}
The right inequality in \eqref{doubleinegalite} remains true even if the sole is not entirely contained in the half space $\left \{  (g,\xi) \in G; \ (\xi-\xi_{0})^{0} >0 \right \}$.\end{rmq}
\subsection{A Wiener criterion and a Poincar\'e cone condition}

Recall a point $(g,\xi)$ is said to be regular with respect to a set $B$ if $\Prob_{(g,\xi)}(T_{B}=0)=1$, where $T_{B}=\inf\{s>0; (g_{s},\xi_{s}) \in B  \}$ is the entrance time in $B$.  

We denote by $B^{r}$ the set of regular points for $B$; by continuity of trajectories we have $\mathring{B} \subset B^{r}\subset \bar{B}$. It is shown in  \cite{Baill06} that $\Op$ admits an adjoint (with respect to the $Haar$ measure of $G$) without zero-order term. Dual processes theory developed in  \cite{BG68} can be applied and we have

\begin{prop}
There exists only one measure $\mu_{B}$ supported by $B^{r}$ such that
\[
\Prob_{(g,\xi)}[T_{B}< + \infty ]= \int G\left ((g,\xi), (g',\xi') \right)\mu_{B}(d(g',\xi')).
\]
\end{prop}
We call capacity of $B$, and denote by $C(B)$, the total mass of $\mu_{B}$. We also  have
\begin{align}
C(B)=\sup \{ \mu(B); \mu \in \mathcal{M}(B), G\mu \leq 1  \}. \label{meseq}
\end{align}
Here $\mathcal{M}(B)$ is the set of finite non-negative measures supported in $B$ and $G\mu= \int G(\cdot,(g',\xi') ) \mu(d(g',\xi'))$.
Fix $\lambda <1$. We denote by
\[
B_{n}:= \left \{ (g,\xi) \in B; \ \lambda^{n+1} \leq \vert (g,\xi) \vert_{(g_{0},\xi_{0})} < \lambda^{n}  \right \}
\]
the homogeneous slices of $B$. Using the double inequality \eqref{doubleinegalite} and a Borel-Cantelli lemma due to J. Lamperti in \cite{Lamp63} we have, as in the elliptic situation ( see \cite{Bass91} ), the following Wiener criterion:

\begin{prop}[Wiener criterion] \label{Wiener}
Let $B$ be a subset of $G$ which is included in  a homogeneous  cone $C_{h}(g_{0},\xi_{0})$.
Then $(g_{0}, \xi_{0})$ is regular for $B$ if and only if $\sum_{n} \lambda^{d(d+3)n} C(B_{n})=\infty$. 
\end{prop}

Let  $H$ be the closure of a domain of $\mathfrak{g}$ included in
\[ \left \{ \left(\begin{array}{ccc}0 &  {}^{t}{(u_{i})} & u_{0} \\(u_{i}) & (u_{ij}) & (u_{i0}) \\0 & 0 & 0 \end{array}\right)  \in \mathfrak{g}; \ u_{0} > 0 \right \}.\]
We denote by $H_{\eps}:= \exp (T_{\eps} (H))$ a ``small'' compact set in the future of $e$. To study the behavior of $C(H_{\eps})$ when $\eps \to 0$ we need the following lemma, where $u_{\eps}$ and $v_{\eps}$ are two arbitrary points of $H_{\eps}$ and $\theta_{u_{\eps}}(v_{\eps}):=T_{1/ \vert v_{\eps} \vert_{u_{\eps}}}(\exp_{u_{\eps}}^{-1}(v_{\eps}))$.

\begin{lem}\label{lem3000}
The following limits exist
\[ \lim_{\eps \to 0}\frac{\vert v_{\eps} \vert_{u_{\eps}}}{\eps}:= \alpha(u,v) > 0 \]
\[ \lim_{\eps \to 0}  \theta_{u_{\eps}}(v_{\eps}):= \beta(u,v) \in \mathfrak{g}-\{0\}. \] 
\end{lem}
\begin{proof}
Let $u,v \in \mathfrak{g}$ be such that $u_{\eps}=\exp(T_{\eps}(u))$ and $v_{\eps}=\exp(T_{\eps}(v))$.
We want to find $w\in \mathfrak{g}$ such that $v_{\eps}=u_{\eps}\times \exp(w)$, \textit{i.e}, $\exp(w)= \exp(-T_{\eps}(u))\exp(T_{\eps}(v)).$ 

By the Campell-Hausdorff formula we obtain
\begin{align}
w= T_{\eps}(v)- T_{\eps}(u) &-\frac{1}{2}[T_{\eps}(u),T_{\eps}(v)]  \label{CH} \\
\nonumber &+\frac{1}{12}( [-T_{\eps}(u),[-T_{\eps}(u),T_{\eps}(v)]]+ [T_{\eps}(v),[T_{\eps}(v),-T_{\eps}(u)]]) + \cdots 
\end{align}

A simple computation gives
\begin{align*}
[T_{\eps}(v),T_{\eps}(u)]= \sum_{i=1}^{d}o(\eps) X_{i}&+ \sum_{1\leq i<j\leq d}\left( \eps^{2}(u_{i}v_{j}-u_{j}v_{i})+o(\eps^2) \right)[X_{i},X_{j}]+\\ &\sum_{i=1}^{d}\left ( \eps^{3}(v_{i}u_{0}-v_{0}u_{i})+ o(\eps^{3}) \right )[X_{i}, X_{0}] + o(\eps^{2})X_{0}.
\end{align*}
 With \eqref{CH} we get
 \[
 w_{i}=\eps (v_{i}-u_{i}) +o(\eps), \quad w_{0}=\eps^{2}(v_{0}-u_{0})+o(\eps^{2}),
 \]
 \[w_{ij}=\eps^{2}(v_{ij}-u_{ij} -\frac{1}{2}(v_{i}u_{j}- u_{i}v_{j}))+o(\eps^{2}),
 \]
 \[
 w_{i0}=\eps^{3}(v_{i0}-u_{i0}-\frac{1}{2}(u_{i}v_{0}-u_{0}v_{i}))+ o(\eps^{3}).
 \]
 By the definition of the homogeneous norm we have $\frac{\vert v_{\eps} \vert_{u_{\eps}}}{\eps}= \frac{\vert w\vert_{e} }{\eps}$ and this quantity converges, when $\eps$ goes to $0$, to the homogeneous norm $\alpha(u,v)\neq 0$ of
 \[
 \left (\begin{matrix}        0 & {}^{t} (v_{i}-u_{i}) & (v_{0}-u_{0}) \\
       (v_{i}-u_{i}) &  (v_{ij}-u_{ij} -\frac{1}{2}(v_{i}u_{j}- u_{i}v_{j}))&(v_{i0}-u_{i0}-\frac{1}{2}(u_{i}v_{0}-u_{0}v_{i})) \\
       0&0&0
    \end{matrix} \right ) \in \mathfrak{g}.
 \]

Moreover, since by definition $\theta_{u_{\eps}}(v_{\eps})=T_{1/\vert w \vert_{e}}(w)$ we have
 \[
 \lim_{\eps \to 0} \theta_{u_{\eps}}(v_{\eps})=  \beta(u,v),
 \]
 where
 \[
   \beta(u,v):=\frac{1}{\alpha(u,v)}\left (\begin{matrix}        0 & {}^{t} (v_{i}-u_{i}) & (v_{0}-u_{0}) \\
       (v_{i}-u_{i}) &  (v_{ij}-u_{ij} -\frac{1}{2}(v_{i}u_{j}- u_{i}v_{j}))&(v_{i0}-u_{i0}-\frac{1}{2}(u_{i}v_{0}-u_{0}v_{i})) \\
       0&0&0
    \end{matrix} \right ) \in \mathfrak{g}.
 \]
\end{proof}
Set
\[
q(H):=\frac{m(H)}{\max_{u\in \partial H} \int_{H} \frac{\mathbf{\Phi}_{e}(\beta(u,v))}{\alpha(u,v)^{Q-2}}m(dv)} \raise 2pt\hbox{.}
\]

\textbf{Notation:} To simplify notations, set $\mathbf{\Phi}_{x}(\cdot):=  \frac{1}{J_{x}(0)}g^{(x)}(0, \cdot) $ where $g^{(x)}$ is the Green function of the tangent process at $x$. We denote by  $m$ the image by $exp^{-1}$ of the  $Haar$ measure on $G$.

\begin{prop}[Capacities of small compact sets] \label{ptitcomp}
\[
\liminf_{\eps \to 0} \frac{C(H_{\eps})}{\eps^{Q-2}}\geq q(H).
\]
\end{prop}
\begin{proof}
Denote by $\nu_{\eps}:=\ind_{H_{\eps}}Haar$ the image of the measure $\nu:= \ind_{H}m$  by the map $\exp \circ T_{\eps}$. 

By \eqref{meseq}, to obtain a lower bound of $C(H_{\eps})$ it is sufficient to get an upper bound for $G\nu_{\eps}$. If $\tau$ denotes the entrance time in $H_{\eps}$, we have, for $(g,\xi) \in  G$: 
\begin{align*}
G\nu_{\eps}(g,\xi)&=\int_{H_{\eps}} G((g,\xi), (g',\xi')) Haar(d(g',\xi')) \\
&= \E_{(g,\xi)}\left [ \int_{\tau}^{+\infty} \ind_{H_{\eps}}(g_{t},\xi_{t})dt \right ]\\
&= \E_{(g,\xi)}\left [   \E_{(g_{\tau},\xi_{\tau})} \left [ \int_{0}^{+\infty} \ind_{H_{\eps}}(g_{t},\xi_{t})dt \right ] \right ],
\end{align*}
hence
\begin{align}
G\nu_{\eps}(g,\xi)=\E_{(g,\xi)}\left [G\nu_{\eps}(g_{\tau},\xi_{\tau}) \right ]. \label{bordsuffit}
\end{align}

Thus it is sufficient to find an upper bound for $G\nu_{\eps}$ on $\partial H_{\eps}$. 
For some $u_{\eps}\in \partial H_{\eps}$, by definition of $\nu_{\eps}$, we have:
\[
G\nu_{\eps}(u_{\eps})= \int_{H} G(u_{\eps}, v_{\eps}) m(dv),
\]
where $v_{\eps}:=\exp(T_{\eps}(v))$. For any $\eta >0$, Theorem \ref{thm1} provides some $\eps_{0}>0$ such that for all $\eps<\eps_{0}$ 
\[
\forall u_{\eps}, v_{\eps} \in H_{\eps}, \quad G(u_{\eps}, v_{\eps})\leq \frac{\eta+ \mathbf{\Phi}_{u_{\eps}}(\theta_{u_{\eps}}(v_{\eps}))}{\vert v_{\eps} \vert_{u_{\eps}}^{Q-2}} \raise 2pt \hbox{.}
\]
Moreover, by Lemma \ref{lem3000}, we can find $\eps_{1}>0$ such that for all $\eps<\eps_{1}$
\[
\forall u_{\eps}, v_{\eps} \in H_{\eps}, \quad \eps^{Q-2} \frac{\eta+ \mathbf{\Phi}_{u_{\eps}}(\theta_{u_{\eps}}(v_{\eps}))}{\vert v_{\eps} \vert_{u_{\eps}}^{Q-2}} \leq \frac{\eta+\mathbf{\Phi}_{e}(\beta(u,v))}{\alpha(u,v)^{Q-2}}+\eta \raise 2pt \hbox{.}
\]
Thus,
\begin{align*}
 \forall u_{\eps} \in \partial H_{\eps} \quad \eps^{Q-2}G\nu_{\eps}(u_{\eps}) &\leq \int_{H} \left ( \frac{\eta+\mathbf{\Phi}_{e}(\beta(u,v))}{\alpha(u,v)^{Q-2}}+\eta \right )m(dv). \\
 &\leq  \eta \times m(H) + \max_{u\in \partial H} \int_{H}  \frac{\eta+\mathbf{\Phi}_{e}(\beta(u,v))}{\alpha(u,v)^{Q-2}} m(dv).
\end{align*}
By \eqref{bordsuffit} we have
\[
\Vert G\nu_{\eps} \Vert_{\infty} \leq \frac{1}{\eps^{Q-2}} \left (  \eta \times m(H) + \max_{u\in \partial H} \int_{H}  \frac{\eta+\mathbf{\Phi}_{e}(\beta(u,v))}{\alpha(u,v)^{Q-2}} m(dv)  \right) \raise 2pt \hbox{.}
\]
By \eqref{meseq}, the previous upper bound provides the following lower bound for the capacity:
\begin{align*}
C(H_{\eps}) &\geq \frac{\eps^{Q-2}\nu_{\eps}(H_{\eps})}{\eta \times m(H) + \max_{u\in \partial H} \int_{H}  \frac{\eta+\mathbf{\Phi}_{e}(\beta(u,v))}{\alpha(u,v)^{Q-2}} m(dv)}\raise 2pt \hbox{.}
\end{align*}
Since $\nu_{\eps}(H_{\eps})=m(H)$, letting $\eta$ go to $0$ we finally obtain
\[
\liminf_{\eps \to 0} \frac{C(H_{\eps})}{\eps^{Q-2}} \geq q(H):= \frac{m(H)}{\max_{u\in \partial H} \int_{H} \frac{\mathbf{\Phi}_{e}(\beta(u,v))}{\alpha(u,v)^{Q-2}}m(dv)}\raise 2pt \hbox{.}
\]

\end{proof}

Propositions \ref{Wiener} and \ref{ptitcomp} above provide a Poincar\'e's cone condition of regularity:

\begin{cor}[Poincar\'e cone condition of regularity]
If  a subset  $B$ contains a homogeneous cone based at $(g_{0},\xi_{0})$  then this point is regular for $B$. 
\end{cor}
\begin{proof}
It suffices to remark that the vertex of an homogeneous cone is regular for it. Indeed, the slices $B_{n}$ are such that $B_{n}=T_{\lambda^{n}}^{(g_{0},\xi_{0})}(B_{0})$ and the capacities $C(B_{n})$ are, by Proposition \ref{ptitcomp}, of order $\lambda^{n(2-Q)}$. Thus $\sum_{n} \lambda^{(Q-2)n} C(B_{n})$ diverges and  we conclude by the Wiener criterion of Proposition \ref{Wiener}. 
\end{proof}

\begin{rmq} 
\begin{itemize}

\item In our case the Green function is not symmetric and we cannot properly mimic the proof of Theorem 6.1 given by Chaleyat -Maurel and Le Gall in \cite{CM-LG} to obtain a precise equivalent of the capacity of small compact sets. As a consequence, it seems to be difficult to obtain more pathwise properties such as limit theorems for the Wiener sausage ( section 7 of \cite{CM-LG}).

\item In \cite{FLJ07} Franchi and Le Jan defined relativistic diffusions with values in the unit tangent bundle, $T^{1}\mathcal{M}$, over any Lorentz manifold $\mathcal{M}$. Roughly speaking, their diffusions are obtained by  ``rolling without slipping'' the space $\Hyp^d \times \R^{1,d}$ over $T^{1}\mathcal{M}$ along some trajectory of Dudley's diffusion. The asymptotic behavior of such diffusions in Robertson-Walker space-times was studied in \cite{Angst09}. These diffusions are projections of diffusions $r_{t}$ in the orthonormal frame bundle $\mathcal{O}\mathcal{M}$ which are solution of the SDE
\[
d r_{t} = \sigma \sum_{i=1}^{d} V_{i}(r_{t}) \circ dB^{i}_{t} + H_{0}(r_{t})dt,
\]
where the $V_{i}$ are vertical vector fields on $\mathcal{O}\mathcal{M}$ corresponding to the infinitesimal action on the fibres of the infinitesimal boost $e_{i}\otimes e_{0} + e_{0}\otimes e_{i} \in so(1,d)$. The vector field $H_{0}$ stands for the horizontal infinitesimal parallel displacement of the frame along the geodesic  started in the direction of  the timelike vector of the frame. 

It can be shown that the tangent process associated to $r_{t}$ is the same that the one associated to $g_{t}$ and we obtain a Wiener criterion and a Poincar\'e cone condition for $r_{t}$ as well. 

\item Since we dispose (see \cite{Bailleul-Raugi}) of a geometric description of  the Poisson boundary of Dudley's diffusion, we can expect a Wiener criterion describing the thinness at the boundary. This question is in the spirit of the works of Ancona who provides, for example in \cite{Anc}, an answer in the case of a Brownian motion in a Cartan-Hadamard manifold. Unfortunately the method used in the present article cannot provide such an asymptotic result, as we have looked here only at the infinitesimal behavior.
\end{itemize}
\end{rmq}

\section{Appendix}

\subsection{Proof of Propositions \ref{loceq} and \ref{propestimapriori}}
\begin{proof}[Proof of Proposition \ref{loceq}]
Introduce first an intermediate pseudo-distance $d_{\B}$ which is a slight modification  of $\vert \cdot \vert$. For this denote as in \cite{N-S-W} by $C_{3}(\delta, \B)$ the class of absolutely continuous maps $\phi: [0,1] \to U$ which  satisfy the differential equation
\[
\phi'(t) = \sum_{ J \in \B} a_{J} X^{J}(\phi(t))
\]
almost everywhere with constants $a_{J}$ such that 
\[
\vert a_{J} \vert < \delta^{\Vert J \Vert}.
\]
Then define for $y, z \in U$
\[
d_{\B}(y,z)= \inf \{ \delta >0 \vert \quad \exists \phi \in C_{3}(\delta,\B) \ \mathrm{such \ that} \ \phi(0)=y, \ \phi(1)=z \}.
\]
Recall that $U$ has been taken small enough for $\varphi_{y}$ be a bijection onto $U$. Thus, for $y, z \in U$ there exists a unique $(a_{J})_{J\in \B}$ such that 
\[
z= \exp \left ( \sum_{J \in \B} a_{J} X^{J} \right )(y),
\]
 \[d_{\B}(y,z)= \max_{J \in \B} \vert a_{J} \vert^{1/\Vert J \Vert}, \]
and
\[
d_{\B}(y,z) \leq \vert z \vert_{y} \leq C d_{\B}(y,z),
\]
where 
\[
C= \left ( \sum_{k=1}^{r} (\mathrm{dim} C_{k} -\mathrm{dim}C_{k-1})^{Q/k}\right)^{1/Q}.
\]
We need to show that $d_{\B}$ is locally equivalent to $\rho$. Since $C_{3}(\delta, \B) \subset C(\delta)$ we deduce 
\[
\forall y, z \in U, \quad \rho(y,z) \leq d_{\B}(y,z).
\]
To show that $d_{\B}$ is locally dominated by $\rho$ we follow the proof of Lemma 2.16 of \cite{N-S-W}; let introduce for that purpose the following notations. For a $n$-tuple $\B'$ of multi-indices of order smaller than $r$ we denote by $\Vert \B' \Vert$ the sum of the orders of its components
\[
\Vert \B' \Vert := \sum_{J \in \B'} \Vert J \Vert. 
\]
Denote also by $\lambda_{\B'}(y)$ the determinant of $(X^{J})_{J \in \B'}$ at $y$. Recall that the $n$-tuple $\B$ has been chosen such that $(X^{J}(y))_{J\in \B}$ is a triangular basis of $T_{y} \mathcal{M}$. This means that if $\B'$ is such that $\lambda_{\B'}(y) \neq 0$ then $\Vert \B  \Vert \leq \Vert \B' \Vert$. Moreover 
\[
M:=\sup_{\B', y} \vert \lambda_{\B'}(y) \vert < +\infty,
\]
where the supremum is taken over all $y\in U$ and over the finite set of the $n$-tuple $\B'$ of multi-indices of order smaller than $r$. We obtain for a $n$-tuple $\B'$ of multi-indices, for $y \in U$ and $\delta \in ]0,1]$
\begin{align*}
\vert \lambda_{\B'}(y) \vert \delta^{\Vert \B' \Vert } \leq M \delta^{\Vert \B' \Vert} \leq M \delta^{\Vert \B  \Vert} \leq \frac{M}{\inf_{z \in U} \vert \lambda_{\B}(z) \vert } \vert \lambda_{\B}(y) \vert \delta^{\Vert \B \Vert}.
\end{align*}
Thus, taking \[t := \frac{\inf_{y\in U} \vert \lambda_{B}(y) \vert}{ M} > 0,\] we have
\[
\forall y \in U, \forall \delta \in ]0,1] \quad \vert \lambda_{\B}(y)\vert  \delta^{\Vert \B \Vert} \geq t \sup_{\B'} \vert \lambda_{\B'} (y) \vert \delta^{\Vert \B' \Vert}.  
\]
Now we apply the lemma 2.16 of \cite{N-S-W} and find $\eta >0$ and $\eps(t) >0$ such that 
\[
\forall y \in U, \forall \delta \in ]0,1], \quad B(y, \eta \eps(t) \delta ) \subset B_{\B}(y, \eps(t) \delta ),
\]
where $B$ and $B_{\B}$ are the balls associated respectively to $\rho$ and $d_{\B}$. Thus if $\rho(y,z)\leq \eta \eps(t) $ then 
\begin{align}
\rho(y,z)= \inf \left \{ n \eps(t) \delta , \ \ z \in B(y, \eta \eps(t) \delta) \right \}\geq \eta \inf \left \{ \eps(t)\delta, \ \ z \in B_{\B}(y, \eps(t) \delta ) \right \} = \eta d_{\B}(y,z). \label{boules}
\end{align}
Then $U$ can be chosen small enough in order that \eqref{boules} be verified for $y, z \in U$. 
\end{proof}
\begin{proof}[Proof of Proposition \ref{propestimapriori}]
It is enough to verify that there exists a positive constant $C$ such that $ \mathrm{Vol}( B( y, \rho(y,z))) \geq C \vert z \vert_{y}^{Q}$ for $y,z\in U$. By Proposition \ref{loceq}, we can find $C' >0$ such that $B(y, \rho(y,z)) \supset B_{\B}(y, C' \vert z \vert_{y} ) $. Then
\begin{align*}
 \mathrm{Vol}( B( y, \rho(y,z))) \geq  \mathrm{Vol} (B_{\B}(y, C' \vert z \vert_{y}))&= \int_{B_{\B}(y, C' \vert z \vert_{y})} d\tilde{z} \\
 &= \int_{\vert u \vert_{n} \leq C' \vert z \vert_{y}} \vert \mathrm{Jac} \ \varphi_{y}(u) \vert du.
\end{align*} 
Since 
\[
\inf_{y \in U, u \in \varphi^{-1}_{y}(U)}  \vert \mathrm{Jac} \ \varphi_{y}(u) \vert >0,
\]
there exists $C'' >0$ such that 
\[
\mathrm{Vol}( B( y, \rho(y,z))) \geq C'' \int_{\vert u \vert_{n} < C' \vert z \vert_{y}} du,
\]
and remarking that $d(T_{\eps} u)= \eps^{Q} du$ we finally obtain, putting $C:= C' C''$,
\[
\mathrm{Vol}( B( y, \rho(y,z))) \geq C \vert z \vert_{y}^{Q}. 
\]
\end{proof}
\subsection{Proof of Proposition \ref{taylor} }
\begin{proof}
According to Theorem 4.1 of \cite{Castell93} (p 234) for $t \leq T$, we have
\[
x_{t}^{\eps}= \exp \left ( \sum_{k=1}^{r} \sum_{L,  \Vert L \Vert =k} \eps^{k}c_{t}^{L}X^{L} \right)(x) + \eps^{r+1}\tilde{R}(\eps,t)
\]
where $\tilde{R}(\eps,t)$ is such that there exists $\alpha, c >0$ for which
\begin{align}
\lim_{\eps \to 0} \Prob \left (\sup_{0\leq s\leq T} \Vert \tilde{R}(\eps,s) \Vert > R, \ T< \tau_{\eps} \wedge T^{\eps}_{U} \right)\leq \exp \left ( -\frac{R^{\alpha}}{cT} \right ) \label{tildeR}
\end{align}
for every $R>c$.
Let us consider the surjection $\psi_{x}$ defined by: 
\[
\psi_{x}( (v_{L})_{\Vert L \Vert \leq r})= \exp \left ( \sum_{L, \Vert L \Vert \leq r} v_{L} X^{L}\right )(x), 
\]
so that 
\begin{align*}
v^{(x, \eps)}_{t}& = T_{1/\eps} \circ \varphi_{x}^{-1} \left ( \psi_{x} \left ( (\eps^{\Vert L \Vert} c_{t}^{L} )_{\Vert L \Vert \leq r} \right )  + \eps^{r+1} \tilde{R}(\eps,t) \right ) \\
&= T_{1 / \eps} \circ \varphi_{x}^{-1} \circ \psi_{x}  \left ( (\eps^{\Vert L \Vert} c_{t}^{L} )_{\Vert L \Vert \leq r} \right ) + \eps \bar{R}(\eps,t) \\
&= \left (\eps^{-\Vert J_{i} \Vert }  (\varphi^{-1}_{x} \circ \psi_{x})_{i} \left ( (\eps^{\Vert L \Vert} c_{t}^{L} )_{\Vert L \Vert \leq r} \right )  \right )_{i=1, \dots ,n } +  \eps \bar{R}(\eps,t),
\end{align*}
where $ \Vert \bar{R}(\eps,t) \Vert \leq Lip( \varphi_{x}^{-1} ) \Vert \tilde{R}(\eps,t)  \Vert$, since $ \Vert T_{1/ \eps} (u) \Vert \leq \eps^{-r} \Vert u \Vert$ for $\eps < 1$. We need to prove that, for $J \in \mathcal{B}$, the Taylor expansion
\begin{align}
(\varphi^{-1}_{x} \circ \psi_{x})_{J} \left ( (\eps^{\Vert L \Vert} c_{t}^{L} )_{\Vert L \Vert \leq r} \right ) = \eps^{\Vert J \Vert} \sum_{L, \Vert L \Vert = \Vert J \Vert} c_{t}^{L} a_{J}^{L}(x) + \eps^{\Vert J \Vert+1} R_{J} (\eps, t) \label{remain}
\end{align}
is such that $R_{J}(\eps, t)$ is bounded in probability as $\tilde{R}(\eps,t)$. For this, we use the following lemma.
\begin{lem}
For $J \in \mathcal{B}$ and $L, \Vert L \Vert \leq r$
\[
\frac{\partial (\varphi^{-1}_{x} \circ \psi_{x})_{J}  }{ \partial v_{L}} (0) = a_{J}^{L}(x) 
\] 
which, by triangularity, is $0$ if $\Vert L \Vert < \Vert J\Vert$. Under the assumption \ref{geomconst} we obtain moreover, for $p\geq 2$ and $L_{1} \dots L_{p}$ with $\sum_{j} \Vert L_{j} \Vert \leq \Vert J \Vert$ , that
\[
\frac{\partial^{p} (\varphi^{-1}_{x} \circ \psi_{x})_{J} }{\partial v_{L_{1}} \cdots \partial v_{L_{p}}}(0)= 0.
\]
\end{lem}
\begin{proof}
Following \cite{B-A.89} p 93-94 a direct computation shows that
\[
\frac{\partial  (\varphi_{x}^{-1} \circ \psi_{x})_{J} }{\partial u_{L}}(0)= a_{J}^{L}(x), 
\]

\[
\frac{\partial^{2} (\varphi_{x}^{-1} \circ \psi_{x})_{J}}{\partial u_{L_{1}} \partial u_{L_{2}}} (0) = X^{L_{1}} (a_{J}^{L_{2}} )(x) + X^{L_{2}} (a_{J}^{L_{1}}) (x)
\]
and by iteration
\[
\frac{\partial^{p} (\varphi_{x}^{-1} \circ \psi_{x})_{J}}{ \partial u_{L_{1}} \cdots \partial u_{L_{p}}} (0)= \sum_{ \sigma \in \mathfrak{S}_{p} } X^{L_{\sigma(1)}} \cdots X^{L_{\sigma(p-1)}} (a_{J}^{L_{\sigma(p)}}) (x).
\]
By Assumption \ref{geomconst}, if $\Vert L \Vert < \Vert J \Vert$ then $a_{J}^{L}=0$ on $U$, and thus if $\sum_{j} \Vert L_{j} \Vert \leq \Vert J \Vert$ and $p\geq 2$ then all the terms in the sum are $0$. 
\end{proof}
By this result we write the Taylor expansion up to the order $\Vert J \Vert +1$, when $\eps \to 0$,
\[
(\varphi^{-1}_{x} \circ \psi_{x})_{J} \left ( (\eps^{\Vert L \Vert} c_{t}^{L} )_{\Vert L \Vert \leq r} \right ) = \sum_{L, \Vert J \Vert \leq \Vert L \Vert \leq r} \eps^{\Vert L \Vert} c_{t}^{L} a^{L}_{J} + \eps^{\Vert J \Vert +1}\hat{R}_{J}(\eps, t)
\]
with $\vert \hat{R}_{J}(\eps, t) \vert \leq \Vert (\eps^{\Vert L \Vert -1} c_{t}^{L})_{\Vert L \Vert \leq r} \Vert^{\Vert J \Vert +1} \times  \Vert D^{\Vert J\Vert+1} (\varphi^{-1}_{x} \circ \psi_{x})_{J} \Vert_{\infty}  $. 
Thus the remainder term $R_{J}(\eps,t)$ defined by \eqref{remain} is 
\[
R_{J}(\eps,t)= \hat{R}_{J}(\eps, t) + \sum_{L, \Vert J \Vert +1 \leq \Vert L \Vert \leq r} \eps^{\Vert L \Vert - \Vert J \Vert -1} c_{t}^{L} a_{J}^{L},
\]
and as in the proof of Lemma 4.1 of \cite{Castell93}, by Properties P1 and P2 p 238 of \cite{Castell93} we deduce that there exist $\alpha_{J}, c_{J} > 0$ such that 
\begin{align}
\forall \eps \leq 1, \quad \forall R \geq c_{J}, \quad \Prob \left [ \sup_{0\leq t \leq T}  \Vert R_{J}(\eps,t) \Vert \leq R; T < T_{U}^{\eps} \right ] \leq \exp \left ( - \frac{R^{\alpha_{J}}}{c_{J}T} \right ). \label{Ri}
\end{align}
Now since $R^{(x)}(\eps,t) = \bar{R}(\eps,t) + ( R_{J}(\eps,t) )_{J \in \mathcal{B}}$ the proposition \ref{taylor} follows from \eqref{tildeR} and \eqref{Ri}.
\end{proof}
\subsection{Proof of Proposition \ref{prop1}}

The proof follows the pattern of Ben Arous and Gradinaru's work \cite{B-A_Grad}. Fix $K$ a compact set of $\R^{n} \setminus\{ 0\}$. The following Proposition ensures, using Ascoli's Theorem, that the family $\{ G^{(x,\eps)}(0,\cdot) \}_{\eps>0}$ is relatively compact for the uniform norm on $K$. Thus, to prove Proposition \ref{prop1} it is enough to show that $G^{(x,\eps)}(0,\cdot)$ converges weakly to $g^{(x)}(0, \cdot)$, which appears to be the unique limit point of $\{ G^{(x,\eps)}(0,\cdot) \}_{\eps>0}$.
\begin{prop}\label{propGxe}
We have
\[
\limsup_{\eps >0} \sup_{u \in K}\vert G^{(x,\eps)}(0,u) \vert < + \infty,
\]
and for $i=1\dots n$
\begin{align}
\limsup_{\eps >0} \sup_{u \in K}  \left \vert \frac{\partial }{\partial u_{i}}G^{(x,\eps)}(0,u) \right \vert < + \infty. \label{Gxe}
\end{align}
\end{prop}
\begin{proof}
By \eqref{rescalingpt} for $u\in K$,  $G^{(x,\eps)}(0,u) = \eps^{Q-2} J_{x}(0) G(x,\varphi_{x} \circ T_{\eps} (u))$ so
\[
\forall u \in K, \quad \vert G^{(x,\eps)}(0,u) \vert \leq C \eps^{Q-2} \vert G(x,\varphi_{x} \circ T_{\eps} (u)) \vert \underset{\mathrm{  Prop \  \ref{propestimapriori}}}{\leq} C' \frac{\eps^{Q-2}}{\vert \varphi_{x} \circ T_{\eps} (u) \vert_{x}^{Q-2} }= \frac{C'}{\vert u\vert_{n}^{Q-2}} \raise 2pt \hbox{.} 
\]
Note that the derivatives in the second inequality of Proposition \ref{propestimapriori} are taken with respect to the second variable of $G$. To simplify notation set $G^{x}(\cdot )=G(x,\cdot)$, $u_{\eps}^{x} = \varphi_{x} \circ T_{\eps} (u)$ and $u_{\eps}= T_{\eps} (u)$.
\begin{align*}
\left \vert \frac{\partial }{\partial u_{i}}G^{(x,\eps)}(0,u) \right \vert &= \vert J_{x}(0) \eps^{Q-2} \frac{\partial}{ \partial u_{i} } \left (  G^{x} \circ \varphi_{x} \circ T_{\eps} (u) \right ) \vert \\
&= \left \vert J_{x}(0) \eps^{Q-2 + \Vert J_{i} \Vert} d_{u_{\eps}^{x}} G^{x} \circ d_{u_{\eps}} \varphi_{x} \left ( \frac{\partial}{\partial u_{i}} \right ) \right \vert.
\end{align*}
By the smoothness of $d_{\bullet} \varphi_{x} \left ( \frac{\partial}{\partial u_{i}} \right )$ (see also Lemma 2.12 of \cite{N-S-W} )
\[
d_{u_{\eps}} \varphi_{x} \left ( \frac{\partial}{\partial u_{i}} \right ) = d_{0} \varphi_{x}\left ( \frac{\partial}{\partial u_{i}} \right ) + O(\eps) = X^{J_{i}} + O(\eps),
\]
where $O(\eps)$ is uniform with respect to $u \in K$.
Thus $d_{u_{\eps}^{x}} G^{x} \circ d_{u_{\eps}} \varphi_{x} \left ( \frac{\partial}{\partial u_{i}} \right ) = X^{J_{i}}G^{x}( u_{\eps}^{x} ) + O(\eps)$ and
\begin{align*}
\left \vert \frac{\partial }{\partial u_{i}}G^{(x,\eps)}(0,u) \right \vert &= C' \eps^{Q-2 + \Vert J_{i} \Vert} \left \vert  X^{J_{i}}G^{x}( u_{\eps}^{x} ) + O(\eps) \right \vert \\
&\underset{\mathrm{Prop \  \ref{propestimapriori}}}{\leq} C' \frac{\eps^{Q-2+ \Vert J_{i} \Vert}(1+ O(\eps))}{\vert u_{\eps}^{x} \Vert_{x}^{Q-2 + \Vert J_{i} \Vert}} \leq \frac{C'(1+ O(\eps))}{\vert u \vert_{n}^{Q-2+ \Vert J_{i} \Vert}}\raise 2pt \hbox{.}
\end{align*}
Taking $\sup_{u\in K}$ then $\limsup_{\eps}$ we obtain \eqref{Gxe}. 
\end{proof}

\begin{prop}\label{prop2}
Let $K$ be  a compact set in $\R^{n}\setminus \{ 0\}$ and $f$ be a smooth function supported on $K$. We have:
\[  
\lim_{\eps \to 0 }\left \vert \int f(u) G^{(x,\eps)}(0,u) du - \int f(u) g^{(x)}(0,u)du \right \vert  =0 .
\]
\end{prop}
\begin{proof}
By definition of the Green functions we have:
\begin{align*}
\left \vert \int f(u) G^{(x,\eps)}(0,u) du - \int f(u) g^{(x)}(0,u)du \right \vert = \left \vert \E \left [ \int_{0}^{\tau_{\eps}} f(v_{t}^{(x,\eps)}) dt -\int_{0}^{+\infty} f(u_{t}^{(x)})dt \right ] \right \vert . 
\end{align*}

Now fixing $T>0$, we can decompose the term $\int_{0}^{\tau_{\eps}}  f(v_{t}^{(x,\eps)}) dt -\int_{0}^{+\infty} f(u_{t}^{(x)})dt $ into: 
\begin{align*}
 \ind_{T\leq \tau_{\eps}} &\int_{0}^{T}\left( f(v_{t}^{(x,\eps)})-f(u_{t}^{(x)})\right )dt + \ind_{T\leq\tau_{\eps}}\int_{T}^{\tau_{\eps}}f(v_{t}^{(x,\eps)})dt -  \int_{T}^{+\infty}f(u_{t}^{(x)})dt  \\
 &+ \ind_{T>\tau_{\eps}}\int_{0}^{\tau_{\eps}}f(v_{t}^{(x,\eps)})dt- \ind_{T>\tau_{\eps}}\int_{0}^{T}f(u_{t}^{(x)})dt.
\end{align*}
Thus we have the inequality:
\begin{align*}
&\left \vert \E \left [ \int_{0}^{\tau_{\eps}}  f(v_{t}^{(x,\eps)}) dt -\int_{0}^{+\infty} f(u_{t}^{(x)})dt \right ] \right \vert  
 \leq \E \left [ \ind_{T\leq\tau_{\eps}} \int_{0}^{T} \vert f(v_{t}^{(x,\eps)})-f(u_{t}^{(x)}) \vert dt \right ] \\ &+ \E \left[ \ind_{T\leq \tau_{\eps}}\int_{T}^{\tau_{\eps}} \vert f(v_{t}^{(x,\eps)})\vert dt \right ] + 
  \E \left [ \int_{T}^{+\infty}\vert f(u_{t}^{(x)})\vert dt\right ]+ 2\Vert f\Vert_{\infty}T\Prob( T\geq \tau_{\eps}).
\end{align*}

Now, taking first the $\limsup_{\eps \to 0}$ and secondly the $\liminf_{T\to \infty}$, Proposition \ref{prop2} follows from the following facts:

\begin{fact}\label{fact1}
For $T>0$ fixed, 
\begin{align}
\lim_{\eps \to 0}\E \left [ \ind_{T\leq\tau_{\eps}} \int_{0}^{T} \vert f(v_{t}^{(x,\eps)})-f(u_{t}^{(x)}) \vert dt \right ] =0. \label{1}
\end{align}
\end{fact}

\begin{fact}\label{fact2}
\[
 \lim_{T\to +\infty} \E \left [ \int_{T}^{+\infty}\vert f(u_{t}^{(x)})\vert dt\right ] =0.
 \]
\end{fact}

\begin{fact}\label{fact3}
\[
\liminf_{T\to +\infty} \limsup_{\eps \to 0} \E \left[ \ind_{T\leq \tau_{\eps}}\int_{T}^{\tau_{\eps}} \vert f(v_{t}^{(x,\eps)})\vert dt \right ] =0.
\]
\end{fact}
\begin{itemize}
\item The proof of Fact \ref{fact1} is a consequence of the Taylor expansion $v_{t}^{(x,\eps)}=u_{t}^{(x)}+\eps R^{(x)}(\eps,t)$, where 
\[
\lim_{\eps \to 0} \Prob \left (\sup_{0\leq s\leq T} \Vert R^{(x)}(\eps,s) \Vert > R, \ T< \tau_{\eps} \wedge T_{U}^{\eps} \right)\leq \exp \left ( -\frac{R^{\alpha}}{cT} \right ).
\]

Indeed, by Proposition \ref{taylor}, for $\eta>0$ we can find $\eps_{0}>0$ and $R>0$ such that for all $\eps< \eps_{0}$
\[
\Prob \left (\sup_{0\leq s\leq T} \Vert R^{(x)}(\eps,s) \Vert > R, \ T< \tau_{\eps} \wedge T_{U}^{\eps} \right)\leq \frac{\eta}{2T \Vert f \Vert}\raise 2pt \hbox{.}
\]
So, decomposing the expectation in \eqref{1} depending on wether the event $\left \{ \underset{0\leq s\leq T}{\sup} \Vert R^{(x)}(\eps,s) \Vert > R, \ T\leq \tau_{\eps} \wedge T_{U}^{\eps} \right \}$ holds or not we obtain 
\begin{align*}
\E \left [ \ind_{T\leq\tau_{\eps}} \int_{0}^{T} \vert f(v_{t}^{(x,\eps)})-f(u_{t}^{(x)}) \vert dt \right ] \leq \eta + \eps T R \Vert Df \Vert_{\infty} + 2T \Vert f \Vert_{\infty} \Prob [ T > T_{U}^{\eps} ].
\end{align*}
As in \cite{Castell93}, p. 238, for $\eps$ sufficiently small 
\[
\Prob [ T > T_{U}^{\eps} ] \leq \sum_{L, \Vert L \Vert \leq r} \exp \left ( -\frac{c_{L}}{\eps^{2\Vert L \Vert} T } \right ).
\]
Hence for all $\eta >0$ we have
\[
\limsup_{\eps \to 0} \E \left [ \ind_{T\leq\tau_{\eps}} \int_{0}^{T} \vert f(v_{t}^{(x,\eps)})-f(u_{t}^{(x)}) \vert dt \right ] \leq \eta,
\]
and Fact \ref{fact1} is proved.

\item To prove Fact \ref{fact2}, it suffices to show that
\begin{align}
\E \left [ \int_{0}^{+\infty} \ind_{B(0,\rho)}(u_{t}^{(x)}) dt \right ] < \infty, \label{integrabilite}
\end{align}
where $B(0,\rho)$ is the ball of radius $\rho$ for the homogenous norm $\vert \cdot \vert_{n}$. Recall (cf proof of Proposition \ref{tangent}) that $\tilde{u}_{t}^{(x)}$ is the image of the diffusion $u_{t}^{(e)}$ which lives in a r-step free nilpotent Lie algebra $\mathcal{G}_{(m,r)}$ by the projection map $p_{x}$. Denote by $\tilde{Q}= \sum_{i=1}^{r} i \times \dim V_{i}$ the homogeneous dimension of $\mathcal{G}_{(m,r)}$ and set $\tilde{T}_{\eps}(\tilde{u})= (\eps^{\Vert K \Vert} \tilde{u}_{K} )_{K \in \mathcal{A}}$ the associated dilations. Recalling that $\tilde{g}$ is the Green function of $\tilde{u}_{t}^{(e)}$ we have
\begin{align}
\E \left [ \int_{0}^{+\infty} \ind_{B(0,\rho)}(u_{t}^{(x)}) dt \right ]  &= \int_{B(0,\rho)} g^{(x)}(0,u)du \underset{by \ \eqref{gtilde}}{=} \int_{B(0,\rho) \times \R^{d}} \tilde{g}(0,(u-M_{x}(v), v))du dv. \label{esp}
\end{align}
The homogeneous character of $\tilde{g}$ with respect to the dilations reads
\[
\tilde{g}\left (0, \tilde{T}_{\eps}(\tilde{u}) \right )= \frac{1}{\eps^{\tilde{Q}-2}} \tilde{g}(0,\tilde{u} ) \raise 2pt \hbox{.}
\]
Taking $\eps^{-1}= \vert \tilde{u} \vert_{n+d}:= \left [ \sum_{k=1}^{r}\left (\sum_{ K \in \mathcal{A}, \Vert K \Vert=k} \tilde{u}_{K}^{2} \right)^{\frac{\tilde{Q}}{2k} }\right ]^{\frac{1}{\tilde{Q}}}$, we obtain the following upper bound for $\tilde{g}$ 
\[
\exists c>0, \forall \tilde{u} \neq 0, \quad \vert \tilde{g}(0, \tilde{u} ) \vert  \leq \frac{c}{\vert \tilde{u} \vert_{n+d}^{\tilde{Q}-2}}\raise 2pt \hbox{.}
\]
Thus with \eqref{esp} we obtain the upper bound
\begin{align}
\E \left [ \int_{0}^{+\infty} \ind_{B(0,\rho)}(u_{t}^{(x)}) dt \right ] \leq c \int_{ B(0, \rho) \times \R^{d}} \frac{1}{\vert (u -M_{x}(v) ,v)  \vert_{n+d}^{\tilde{Q}-2}} du dv. \label{machin}
\end{align}
For $k=1, \dots, r$, set 
\[
(u,v)_{k}:= \left ( (u_{K})_{K \in \mathcal{B}},(v_{K})_{K \in \mathcal{A}\setminus \mathcal{B}}\right )_{ \Vert K \Vert =k} \in \R^{\mathrm{dim} V_{k}}.
\]
The following endomorphism of $\R^{\mathrm{dim} V_{k}}$
\begin{align*}
(u,v)_{k}  \longmapsto  \left ( \left (u_{J} -\sum_{\underset{\Vert K \Vert =k}{K \in \mathcal{A} \setminus \mathcal{B}} } a_{J}^{K} v_{K} \right)_{J \in \mathcal{B}, \Vert J \Vert = k }, (v_{K})_{K \in \mathcal{A}\setminus \mathcal{B},\Vert K \Vert =k } \right )
\end{align*}
which is the $k^{\mathrm{th}}$ block of $(u,v) \mapsto (u-M_{x} (v),v)$, is an isomorphism. Therefore there exists a constant $C_{k}$ such that 
\[
\forall (u,v)_{k}, \quad C_{k}\left ( \sum_{\underset{\Vert J \Vert =k}{J \in B}} (u_{J})^{2} + \sum_{\underset{\Vert K \Vert =k}{K \in \mathcal{A}\setminus \mathcal{B}} }(v_{K})^{2} \right ) \leq \sum_{\underset{\Vert J \Vert =k}{J \in B}} \left ( u_{J} -\sum_{\underset{\Vert K \Vert =k}{K \in \mathcal{A} \setminus \mathcal{B}} } a_{J}^{K} v_{K} \right )^{2} +  \sum_{\underset{\Vert K \Vert =k}{K \in \mathcal{A}\setminus \mathcal{B}} }(v_{K})^{2}.
\]
Thus we can find a constant $C=\min \{ C_{k}, k=1, \dots, r \} >0$ such that
\[
\forall (u,v)\in \R^{n+d}, \quad C \vert (u,v) \vert_{n+d} \leq \vert (u-M_{x}(v) , v) \vert_{n+d}.
\]
Moreover, Lemma (A.7) of \cite{B-A_Grad} states exactly that 
\[
\int_{B(0, \rho) \times \R^{d}} \frac{1}{\vert (u,v) \vert_{d}} dudv < +\infty,
\] 
and we deduce, by \eqref{machin}, that
\[
\E \left [ \int_{0}^{+\infty} \ind_{B(0,\rho)}(u_{t}^{(x)}) dt \right ] < +\infty.
\] 
\item The proof of Fact \ref{fact3} is very similar to the proof of Proposition (4.1) of \cite{B-A_Grad}, and we prove two lemmas to conclude.
 Denote by $\mu_{T}^{(x,\eps)}$ the measure whose density is $\ind_{T<\tau_{\eps}}$ with respect to the law of $v_{T}^{(x,\eps)}$. By the Markov property we have:
\[
\E_{0}\left [ \ind_{T<\tau_{\eps}} \int_{T}^{\tau_{\eps}} \vert f(v_{t}^{(x,\eps)}) \vert dt \right ]=\int_{T_{1/\eps}\circ \varphi_{x}^{-1}(U)} d\mu_{T}^{(x,\eps)}(u) \int_{T_{1/\eps}\circ \varphi_{x}^{-1}(U)} G^{(x,\eps)}(u,v)\vert f(v) \vert dv.
\]
For $u,v \in T_{1/\eps} \circ \varphi_{x}^{-1}(U)$, we set $u_{\eps}^{x}:=\varphi_{x}\circ T_{\eps}(u)$ and $v_{\eps}^{x}:=\varphi_{x}\circ T_{\eps}(v)$. \footnote{Be carefull: $u^{x}_{\eps}$ is a vector of $\R^{n}$ and  $u_{\eps}^{(x)}$ denotes the tangent process at time $\eps$.}
 
By \eqref{zib} we have:
 \begin{align*}
 G^{(x,\eps)}(u,v)=\eps^{Q-2}J_{x}(T_{\eps}(v))G(u_{\eps}^{x},v_{\eps}^{x}),
 \end{align*}
 and using Proposition \ref{propestimapriori} we can find $C>0$ such that for all $\eps>0$:
 \[
 \int_{T_{1/\eps}\circ \varphi_{x}^{-1}(U)} G^{(x,\eps)}(u,v)\vert f(v) \vert dv\leq \int_{B(0,\rho)}\frac{C\eps^{Q-2}}{\vert v_{\eps}^{x}\vert^{Q-2}_{u_{\eps}^{x}}}dv,
 \] 
where $\rho$ is large enough, so that the support of $f$ be included in $B(0,\rho)$.
 
 As in \cite{B-A_Grad} we show:
 \begin{lem}\label{lem1}
For all $R>0$ there exist $\eps_{0}>0$ and $c>0$ such that for all $\eps<\eps_{0}$ and $u\in \R^{n}$ such that $\Vert u \Vert \geq R$ we have:
\begin{align*}
 \int_{B(0,\rho)}\frac{\eps^{Q-2}}{\vert v_{\eps}^{x}\vert_{u_{\eps}^{x}}^{Q-2}}dv\leq c.
 \end{align*}
 Moreover,
 \begin{align*}
 \lim_{\Vert u\Vert \to \infty} \int_{B(0,\rho)}\frac{\eps^{Q-2}}{\vert v_{\eps}^{x}\vert_{u_{\eps}^{x}}^{Q-2}}dv =0,
 \end{align*}
 uniformly with respect to $\eps$.
 \end{lem}
 \begin{proof}[Proof of Lemma \ref{lem1}]
 For $u=0$, this means $u_{\eps}^{x}=x$, $\frac{\eps^{Q-2}}{\vert v_{\eps}^{x}\vert_{u_{\eps}^{x}}^{Q-2}}=\frac{1}{\vert \varphi_{x}(v) \vert_{x}^{Q-2}}=\frac{1}{\vert v \vert_{n}^{Q-2}} \raise 2pt \hbox{.}$
 
Lemma (A-1) of \cite{B-A_Grad} states, under assumption \ref{ass2}, that $\int_{B(0,\rho)} \frac{dv}{\vert v \vert_{d}^{Q-2}}$ is bounded by a constant depending only on $\rho$.
As the family of diffeomorphisms $\varphi_{y}^{-1}$ depends smoothly on $y\in U$, we can find a neighbourhood $U_{x}$ of $x$, a constant $C>0$ and some $\eps_{0}>0$ such that for all $y\in U_{x}$ and all $\eps<\eps_{0}$ we have $B(0,\eps \rho) \subset \varphi_{x}^{-1}(U) \cap \varphi_{y}^{-1}(U)$ and for all $w\in B(0,\eps \rho)$ we have $\vert \varphi_{y}^{-1}\circ \varphi_{x}(w) \vert_{n} \geq C \vert w \vert_{n}$. Thus we obtain
\begin{align}
\sup_{y\in U_{x}} \int_{B(0,\rho)} \frac{\eps^{Q-2}}{\vert v_{\eps}^{x}\vert_{y}^{Q-2}}dv &= \sup_{y\in U_{x}} \int_{B(0,\rho)} \frac{\eps^{Q-2}}{\vert \varphi_{y}^{-1}(v_{\eps}^{x})\vert_{n}^{Q-2}}dv 
=\sup_{y\in U_{x}} \int_{B(0,\rho)} \frac{\eps^{Q-2}}{\vert \varphi_{y}^{-1}\circ \varphi_{x}\circ T_{\eps}(v)\vert_{n}^{Q-2}}dv \notag \\
&\leq   \tilde{C}\int_{B(0,\rho)}\frac{\eps^{Q-2}}{\vert T_{\eps}(v)\vert_{n}^{Q-2}}dv \leq \tilde{C} \int_{B(0,\rho)} \frac{dv}{\vert v \vert_{n}^{Q-2}} \label{c} \raise 2pt \hbox{.}
\end{align}
Now choose $\eps_{1}>0$ such that for all $\eps<\eps_{1}$ and for all $u \in \R^{n}$ with $\Vert u\Vert \leq R$, we have $u_{\eps}^{x}\in U_{x}$. Using \eqref{c}, there is a constant $c$ such that for all $\eps<\min (\eps_{0}, \eps_{1})$ and for all $u \in \R^{n}$  with $\Vert u \Vert \leq R$
\[
\int_{B(0,\rho)}\frac{\eps^{Q-2}}{\vert v_{\eps}^{x}\vert_{u_{\eps}^{x}}^{Q-2}}dv\leq c.
\]

This is the first point of the lemma.
The second point will follow from \eqref{intri} and \eqref{compnormeucli} of Proposition \ref{prop4}.
For $\Vert u \Vert$ large enough such that $\frac{1}{c_{0}}\vert u_{\eps}^{x} \vert_{x}- \vert v_{\eps}^{x} \vert_{x} >0$ we have
\begin{align*}
 \int_{B(0,\rho)}\frac{\eps^{Q-2}}{\vert v_{\eps}^{x}\vert_{u_{\eps}^{x}}^{Q-2}}dv \leq \int_{B(0,\rho)}\frac{\eps^{Q-2}}{\left ( \frac{1}{c_{0}}\vert u_{\eps}^{x} \vert_{x}- \vert v_{\eps}^{x} \vert_{x} \right ) ^{Q-2}} dv 
 &\leq \int_{B(0,\rho)}\frac{1}{\left ( \frac{1}{c_{0}}\vert \varphi^{-1}(u) \vert_{x}- \vert \varphi^{-1}(v) \vert_{x} \right ) ^{Q-2}} dv \\
 &\leq \int_{B(0,\rho)}\frac{1}{\left ( \frac{c'}{c_{0}}\Vert u \Vert- \rho \right ) ^{Q-2}} dv.
\end{align*}
This inequality ensures that the convergence is uniform in $\eps$.
 \end{proof}
 
Let us return to the proof of Fact \ref{fact3}.
 
By the previous lemma we obtain
\begin{align*}
\E_{0}\left [ \ind_{T<\tau_{\eps}} \int_{T}^{\tau_{\eps}} \vert f(v_{t}^{(x,\eps)}) \vert dt \right ] &\leq c \mu_{T}^{(x,\eps)}(B(0,R)) +\underset{\underset{R\to +\infty}{\longrightarrow }0}{\underbrace{\sup_{u, \Vert u \Vert \geq R}  \int_{B(0,\rho)}\frac{C\eps^{Q-2}}{\vert v_{\eps}^{x}\vert^{Q-2}_{u_{\eps}^{x}}}dv}}.
\end{align*}

To end the proof of Fact \ref{fact3} it remains to prove the following fact.

\begin{lem}\label{lem2}
For all $R>0$,
\begin{align*}
\liminf_{T\to +\infty} \limsup_{\eps \to 0} \mu_{T}^{(x,\eps)}(B(0,R))=0.
\end{align*}
\end{lem}
\begin{proof}[Proof of Lemma \ref{lem2}]
By definition of $\mu_{T}^{(x,\eps)}$ we have, \[\mu_{T}^{(x,\eps)}(B(0,R))=\E_{0} \left[ \ind_{T<\tau_{\eps}} \ind_{B(0,R)}(v_{T}^{(x,\eps)}) \right ].\]
Let  $\chi$ be a smooth function which is equals to $1$ on $B(0,R)$ and which is supported on $B(0,R+1)$. We have
\[
\mu_{T}^{(x,\eps)}(B(0,R)) \leq \E_{0} \left[ \ind_{T<\tau_{\eps}} \chi(v_{T}^{(x,\eps)})\right].
\]
Using the Taylor expansion of $v_{T}^{(x,\eps)}$, as in the proof of the fact \ref{fact1}, we obtain:
\[
\E_{0} \left[ \ind_{T<\tau_{\eps}} \chi(v_{T}^{(x,\eps)})\right] \underset{\eps \to 0}{\longrightarrow} \E_{0} \left[  \chi(u_{T}^{(x)})\right].
\]
Moreover $ \E_{0} \left[ \chi(u_{T}^{(x)})\right] \leq  \E_{0} \left[ \ind_{B(0,R+1)}(u_{T}^{(x)})\right]$ and  $t\mapsto  \E_{0} \left[ \ind_{B(0,R+1)}(u_{t}^{(x)})\right]$ is a positive bounded function which is integrable according to \eqref{integrabilite}. Thus we obtain \[\liminf_{T\to +\infty} \E_{0} \left[ \ind_{B(0,R+1)}(u_{T}^{(x)})\right]=0.\]
\end{proof}
\end{itemize}
\end{proof}

\textbf{Acknowledgements}. The author wishes to thank J. Franchi and M. Gradinaru for reading this article and the anonymous referee for his comments and suggestions. 
\bibliography{bibli}
\bibliographystyle{plain}
\end{document}